\documentclass[12pt]{amsart}
\usepackage{amssymb}
\usepackage{eucal}
\usepackage{bm}
\usepackage[all,cmtip]{xy}
\usepackage{color}
\usepackage[bookmarks=false]{hyperref}
\usepackage{multirow,bigdelim}
\usepackage{tikz}
\usetikzlibrary{matrix,arrows}
\usepackage{commath}
\usepackage{enumerate, tikz-cd}

\allowdisplaybreaks

\newtheorem{theorem}[equation]{Theorem}
\newtheorem{lemma}[equation]{Lemma}
\newtheorem{proposition}[equation]{Proposition}
\newtheorem{corollary}[equation]{Corollary}

\theoremstyle{definition}
\newtheorem{definition}[equation]{Definition}

\theoremstyle{remark}
\newtheorem{remark}[equation]{Remark}

\numberwithin{equation}{section}

\allowdisplaybreaks[1]

\newcommand{\ds}{\displaystyle}

\def\XXint#1#2#3{{\setbox0=\hbox{$#1{#2#3}{\int}$ }
\vcenter{\hbox{$#2#3$ }}\kern-.6\wd0}}

\DeclareMathOperator{\MSpec}{MSpec}

\DeclareMathOperator{\mspec}{MSpec}

\newcommand{\Fq}{\mathbb{F}_{q}}

\newcommand{\bv}{\mathbf{v}}

\newcommand{\bz}{\mathbf{z}}

\newcommand{\cM}{\mathcal{M}}

\newcommand{\cO}{\mathcal{O}}

\newcommand{\cC}{\mathcal{C}}
\newcommand{\cR}{\mathcal{R}}

\newcommand{\cF}{\mathcal{F}}

\DeclareMathAlphabet{\matheur}{U}{eur}{m}{n}

\newcommand{\fm}{\mathfrak{m}}

 \DeclareMathOperator{\Lie}{Lie}
 \DeclareMathOperator{\GL}{GL}
\DeclareMathOperator{\Mat}{Mat} 
 
 \DeclareMathOperator{\Gal}{Gal}

\DeclareMathOperator{\Id}{Id}

\DeclareMathOperator{\Vol}{Vol}
\DeclareMathOperator{\Fitt}{Fitt}

\newcommand{\OX}{\mathcal{O}}

\newcommand{\power}[2]{{#1 [\![ #2 ]\!]}}
\newcommand{\laurent}[2]{{#1 (\!( #2 )\!)}}

\definecolor{ForestGreen}{rgb}{0.0, 0.5, 0.0}

\newcommand{\C}{\ensuremath \mathbb{C}}
\newcommand{\Z}{\ensuremath \mathbb{Z}}

\newcommand{\F}{\ensuremath \mathbb{F}}

\newcommand{\bU}{\mathbf{U}}

\newcommand{\cS}{\mathcal{S}}
\newcommand{\cU}{\mathcal{U}}

\newcommand{\isom}{\ensuremath \cong}
\newcommand{\inv}{\ensuremath ^{-1}}

\DeclareMathOperator{\Exp}{Exp}
\DeclareMathOperator{\Log}{Log}

\newcommand{\twist}{^{(1)}}

\newcommand{\twistk}[1]{^{(#1)}}

\makeatletter
\@namedef{subjclassname@2010}{\emph{2020} Mathematics Subject Classification}
\makeatother

\def\XXint#1#2#3{{\setbox0=\hbox{$#1{#2#3}{\int}$ }
\vcenter{\hbox{$#2#3$ }}\kern-.6\wd0}}

\title[An Equivariant Tamagawa Number Formula for abelian $t$-Modules]{An Equivariant Tamagawa Number Formula \linebreak for Abelian $t$-Modules and Applications}

\author{Nathan Green}
\address{Dept. of Mathematics and Statistics, Louisiana Tech University, Ruston, LA, 71270, USA}

\email{ngreen@latech.edu}

\author{Cristian D. Popescu}
\address{Dept. of Mathematics, University of California at San Diego, San Diego, CA, 92093, USA}

\email{cpopescu@ucsd.edu}



\date{}

\begin{document}

\begin{abstract} We fix motivic data $(K/F, E)$ consisting of a Galois extension $K/F$ of characteristic $p$ global fields with arbitrary abelian Galois group $G$ and an abelian $t$--module $E$, defined over a certain Dedekind subring of $F$. For this data, one can define a
$G$--equivariant motivic $L$--function $\Theta_{K/F}^E$. We refine the techniques developed in \cite{FGHP20} and prove an equivariant Tamagawa number formula for appropriate Euler product completions of the special value $\Theta_{K/F}^E(0)$ of this equivariant $L$--function. 
This extends our results in \cite{FGHP20} from the Drinfeld module setting to the $t$--module setting. 
As a first notable consequence, we prove a $t$--module analogue of the classical (number field) Refined Brumer--Stark Conjecture, relating a certain $G$--Fitting ideal
of the $t$--module analogue of Taelman's class modules \cite{T12} to the special value $\Theta_{K/F}^E(0)$ in question. As a second consequence, we prove formulas for the values $\Theta_{K/F}^E(m)$, at all positive 
integers $m\in\Bbb Z_{\geq 0}$, when $E$ is a Drinfeld module. This, in turn, implies a Drinfeld module analogue of the classical Refined Coates--Sinnott Conjecture relating $\Theta_{K/F}^E(m)$ to the Fitting ideals of certain Carlitz twists of Taelman's class modules, suggesting a strong analogy between these twists and the even Quillen $K$--groups of a number field. 
In an upcoming paper, these consequences will be used to develop an Iwasawa theory for the $t$--module analogues of Taelman's class modules.
\end{abstract}

\date{\today}
\maketitle

\section{Introduction}
\subsection{The arithmetic data}
Let $q=p^r$ for some fixed prime $p\in \Z$ and let $A = \F_q[t]$ and $k = \F_q(t)$.  Let $F/k$ be a finite, separable extension and $K/F$ be a finite, Galois extension with abelian Galois group $G = \Gal(K/F)$. If $L$ is a field of characteristic $p$, we let $\overline L$ be its separable closure and $L^{\rm alg}$ its algebraic closure. For simplicity, we assume that $K\cap  \overline{\F}_q = F\cap  \overline{\F}_q = \F_q$.  We let $\cO_F$ and $\cO_K$ be the integral closures of $A$ in $F$ and $K$, respectively.

For a prime $v\in{\rm MSpec}(\mathcal O_F)$, we denote by $I_v$ its inertia group in $G$. We let $\tilde \sigma_v$ denote a Frobenius automorphism associated to $v$ in the Galois group over $F$ of the maximal subfield of $\overline F$ which is unramified at $v$.
If $v$ is unramified in $K/F$, we denote by $\sigma_v$ its Frobenius automorphism in $G$. Obviously, in that case the restriction of $\widetilde\sigma_v$ to $K$ is $\sigma_v$.

We let $k_\infty = \laurent{\F_q}{1/t}$ be the completion of $\F_q(t)$ at the infinite place (corresponding, as usual, to the valuation on $k$ of uniformizer $1/t$), and set $\C_\infty$ to be a completion of $k_\infty^{\rm alg}$ at the infinite place.  Also, we will let $K_\infty:=K\otimes_k k_\infty$ and $F_\infty:=F\otimes_k k_\infty$. These are the direct sums of the completions of $K$ and $F$ at all the
primes in these respective fields sitting above the infinite prime of $k$. As usual, we view them as topological algebras endowed with the direct sum topologies.

\subsection{The relevant $t$--modules}
For an $\F_q$-algebra $R$, we let $\tau$ denote the $q$-power Frobenius of $R$ and let $R\{\tau\}$ denote the twisted polynomial ring in $\tau$, subject to the relation
\[\tau\cdot x=x^q\cdot\tau, \text{  for all }x\in R.\]
If $R$ is commutative, then ${\text M}_n(R)$ denotes the ring of $n\times n$ matrices with entries in $R$.

\begin{definition}
A $t$-module $E$ of dimension $n$ defined over $\cO_F$ is given by an $\F_q$-algebra morphism (called the structural morphism of $E$ in what follows)
\[\phi_E:A \to {\text M}_n(\cO_F)\{\tau\},\]
such that
\[\phi_E(t) = d_E[t]\tau^0 + M_1 \tau + \dots + M_\ell \tau^\ell,\quad M_i \in {\text M}_n(\cO_F),\]
where $d_E[t] = t\cdot \Id_n + N$, for $N \in {\text M}_n(\cO_F)$ a nilpotent matrix. \end{definition}

\begin{remark} We note that $t$-modules of dimension 1 are simply Drinfeld modules. \end{remark}

Any $t$-module $E$ as above gives rise to a functor (also denoted by $E$, abusively)
\[E: \left ({\text M}_n(\cO_F)\{\tau\}-\text{modules}\right ) \to \left (A-\text{modules}\right ),\quad B\mapsto E(B).\]
In this way $E$ gives an $A$-module structure to any ${\text M}_n(\cO_F)\{\tau\}$-module.\\

It is easily seen that the map $d_E[\cdot]:A\to {\text M}_n(\cO_F)$, sending $a\in A$ to the $\tau$--constant term of $\phi_E(a)$, is a ring morphism.  This ring morphism gives rise to a new functor
\[\Lie_E: \left ({\text M}_n(\cO_F)\{\tau\}-\text{modules}\right ) \to \left (A-\text{modules}\right ),\quad B\mapsto \Lie_E(B).\]
Every $a\in A$ acts on $B$ via $d_E[a]=a\cdot {\rm Id}_n+N_a$, where $N_a\in M_n(\mathcal O_F)$ is a nilpotent matrix.  We will refer to this action as the $\Lie$ action of $E$ on $B$ and will denote the ensuing $A$--module by $\Lie_E(B)$.

\begin{remark} When restricted to the subcategory of ${\text M}_n(\cO_F)\{\tau\}[G]-\text{modules}$ (i.e. ${\text M}_n(\cO_F)\{\tau\}$--modules on which
$G$ acts ${\text M}_n(\cO_F)\{\tau\}$--linearly), the above functors $E$ and $\Lie_E$ take values in the category of $A[G]$--modules. Relevant examples of modules $B$ belonging to this subcategory are given by
$K^n$, $O_K^n$, $(\mathcal O_K/v)^n$, $(K_\infty)^n$,
where $v$ is a finite prime in $\mathcal O_F$.
\end{remark}

For simplicity, if $M$ is an $\mathcal O_F\{\tau\}$--module (which makes $M^n$ an ${\text M}_n(\mathcal O_F)\{\tau\}$--module in the obvious fashion), in what follows we let $E(M)$ and $\Lie_E(M)$ denote
$E(M^n)$ and $\Lie_E(M^n)$, respectively. The following (otherwise elementary) result will be very useful in what follows.

\begin{lemma}\label{A-free-Lemma} Let $M$ be a $\mathcal O_F\{\tau\}$--module, which is finitely generated as an $\mathcal O_F$--module.
\begin{enumerate}
\item
If $M$ is projective of (necessarily locally constant) rank $m$ as an $\mathcal O_F$--module, then $\Lie_E(M)$ is a free $A$--module of rank $mn\cdot[F:k]$.
\item If $M$ is an $\mathcal O_F\{\tau\}[G]$--module, which is a projective $\mathcal O_F[G]$--module, then $\Lie_E(M)$ is a projective $A[G]$--module. Moreover, if $M$ is projective of constant rank $1$
over $\cO_F[G]$, then $\Lie_E(M)$ is projective of constant rank $n\cdot[F:k]$ over $A[G]$.
\end{enumerate}
\end{lemma}
\begin{proof}
(1) Observe that there is an integer $\ell\gg 0$, such that $N_t^{q^\ell}=0$. It is clear that for such an $\ell$, we have $N_a^{q^\ell}=0$, for all $a\in A$. Consequently, if we let $A^{(\ell)}:=\Bbb F_q[t^{q^\ell}]$, then the identity map gives an isomorphism
of $A^{(\ell)}$--modules
$$M^n\isom\Lie_E(M),$$
where the left--side is endowed with the straight scalar multiplication action and the right with the ${\rm Lie}_E$--action. Now, (1) above follows from the fact that $A$ is a PID, $\mathcal O_F$ is a free $A$--module of rank $[F:k]$, and $A$ is a free $A^{(\ell)}$--module (of rank
$q^\ell$.)\\

(2) By Corollary 7.1.7 of \cite{FGHP20}, since $M$ is a projective $\mathcal O_F[G]$--module, $M$ is $\mathcal O_F$--projective and $G$-c.t. ($G$--cohomologically trivial). Therefore, $\Lie_E(M)$ is $A$--free (by (i)) and $G$--c.t. (Note that the $\Bbb Z[G]$--module structures of $M^n$ and $ \Lie_E(M)$ are identical. Therefore, if $M$ is $G$--c.t., then $M^n$ is $G$--c.t. and $\Lie_E(M)$ is $G$--c.t.)
Therefore (by loc.cit.) $\Lie_E(M)$ is $A[G]$--projective. The $A[G]$--rank equality follows easily by looking at $A^{(\ell)}[G]$ instead (see proof of (1) above.)
\end{proof}

\begin{remark} [Extending the Lie action to $k_\infty$]\label{kinfty-free-Remark}
Observe that the matrix $d_E[a]$ is invertible in $M_n(F)$, for all $a\in A\setminus\{ 0\}$. Also, for all valuations $v_\infty$ of $F$ above $\infty$, we  have
\[\min_{i,j}\left (v_\infty\left ((d_E[t]^{-m})_{i,j}\right )\right ) \to \infty, \quad \text{as }m\to \infty.\]
This implies the map $d_E[\cdot]$ extends naturally to a ring homomorphism
$$d_E:k_\infty \to {\text M}_n(F_\infty).$$
In particular, via this map $\Lie_E(F_\infty)$ is
naturally endowed with a $k_\infty$--module structure (extending its $A$--module structure) and $\Lie_E(K_\infty)$ is endowed with a $k_\infty[G]$--module structure
(extending its $A[G]$--module structure.) An analogue of Lemma \ref{A-free-Lemma}(1) shows that
$${\rm dim}_{k_\infty}\Lie_E(K_\infty)={\rm dim}_{k_\infty}K_\infty^n=n\cdot [K:k].$$
Also, by Hilbert's normal basis theorem, $K\isom F[G]$, as $F[G]$--modules. Therefore, from the definitions, we have an isomorphism of $F_\infty[G]$--modules and $k_\infty[G]$--modules, respectively
$$K_\infty\isom F_\infty[G], \qquad K_\infty\isom k_\infty[G]^r,$$
where $r=[F:k]$. Therefore, $K_\infty$ is a rank $1$, free $F_\infty[G]$--module. 
Consequently, by an analogue of Lemma \ref{A-free-Lemma}(2), in the proof of which one replaces $A^{(\ell)}$ with $k_\infty^{(\ell)}:=\Bbb F_q((t^{-q^\ell}))$, we have the following.
\end{remark}
\begin{lemma}\label{kinfty[G]-free-Lemma} The $k_\infty[G]$--module $\Lie_E(K_\infty)$ is free of rank $rn$, where $r=[F:k]$.
\end{lemma}
\begin{proof} See the remark above.
\end{proof}

\begin{definition} [The exponential map of $E$, see \cite{And86}]
Associated to the $t$-module $E$, there is a uniquely defined power series $\Exp_E \in {\rm Id}_n+\bz\cdot {\rm M}_n(\cO_F)[[\bz]]$, such that
\[\Exp_E(d_E[a] \bz) = \phi_E(a)(\Exp_E(\bz)),\quad \text{for all }a\in A.\]
This series converges on $\Bbb C_\infty^n$, defines an analytic function $\Exp_E:\C_\infty^n\to \C_\infty^n$ which is surjective if $E$ is uniformizable (see \cite[\S 2.2]{And86}). The series $\Exp_E$ is called the exponential of $E$ and its formal inverse $\Log_E$ is called the logarithm of $E$. The series $\Log_E$ converges only on some finite polydisc inside $\C_\infty^n$.
\end{definition}

\begin{remark}Note that the definition and uniqueness of $\Exp_E$ turn it into an analytic, open $A[G]$--linear morphism
$$\Exp_E: \Lie_E(K_\infty)\to E(K_\infty)$$
\end{remark}

\subsection{The associated special $L$--values}\label{L-values-section} From now on, we assume that $E$ is an {\it abelian $t$--module.} (See Definition 5.4.12 in \cite{Goss}.) The main reason for this assumption is to ensure that the $v_0$--adic Tate modules $T_{v_0}(E)$ are $A_{v_0}$--free of finite rank (equal to the rank of $E$), for all primes $v_0\in{\rm MSpec}(A)$.
In particular, if $E$ is a Drinfeld module or a tensor product of Drinfeld modules, then $E$ is abelian. (See \S5 in \cite{Goss} for all these facts.)\\
 
Now, we follow the ideas of \cite{FGHP20},  to associate to the data $(K/F, E)$ an Euler--incomplete, $G$--equivariant $L$--function
$$\Theta_{K/F, S_0}^E: \Bbb S_\infty^+\to \Bbb C_\infty[G],$$
where $S_0$ is the subset of ${\rm MSpec}(\mathcal O_F)$ consiting of all the finite primes of bad reduction for $E$ and those which are wildly ramified in $K/F$, $\Bbb S_\infty:=\Bbb C_\infty^\times\times\Bbb Z_p$ is Goss's ``complex plane'', and $\Bbb S_\infty^+:=\{(x,y)\in\Bbb S_\infty\mid |x|_\infty\geq 1\}$ is a``half--plane'' which contains $\Bbb Z_{\geq 0}$, under the Goss embedding $\Bbb Z\to \Bbb C_\infty$ given by $n\to (t^n, n)$. Here, $|\cdot|_\infty$ is the usual normalized abosolute value on $\Bbb C_\infty$.
In order to define the $L$--function above, as in \cite{FGHP20}, for $v_0\in {\rm MSpec}(A)$, we let 
$$H^1_{v_0}(E):={\rm Hom}_{A_{v_0}}(T_{v_0}(E), A_{v_0}), \quad H^1_{v_0}(E, G):=H^1_{v_0}(E)\otimes_{A_{v_0}}A_{v_0}[G]$$ 
and endow these free $A_{v_0}$--modules with the contravariant and the diagonal  $G_F$--action, respectively. Now, we let  $v\in{\rm MSpec}(\mathcal O_F)\setminus S_0$, with $v\nmid v_0$ and remind the reader that the $G_F$--representation $H^1_{v_0}(E)$ is unramified at $v$. (See \cite{Goss}, \S8.6.) This means that the $G_F$--representation $H_{v_0}(E)\otimes_{A_{v_0}}A_{v_0}[G]^{I_v}$ is unramified at $v$. Since this last $A_{v_0}[G]$--module is finitely generated and projective (see \cite{FGHP20} for a proof), one can consider the polynomial
$$P_v^{\ast, G, E}(X):={\rm det}_{A_{v_0}[G]}(1-X\cdot\widetilde\sigma_v\, \mid H_{v_0}^1(E)\otimes_{A_{v_0}}A_{v_0}[G]^{I_v}).$$
As in loc.cit. (see \S8.6 in \cite{Goss} as well) one shows that this polynomial does not depend on $v_0$ and has coefficients in $A[G]$, for every $v$ as above. 
For any finite set $S\subseteq {\rm MSpec}(\mathcal O_F)$, containing $S_0$, we let 
\begin{equation}\label{incomplete-L-definition}\Theta_{K/F, S}^E(s):=\prod_{v\not\in S}P_{v}^{\ast, G, E}(Nv^{-s})^{-1},\end{equation}
where $Nv$ is the monic generator of the ideal $N_{\mathcal O_F/A}(v)$ and $s\to Nv^{-s}$ is Goss's exponential function defined on $\Bbb S_\infty$. (See \cite{Goss}, \S8.1.)  It is important to note that under Goss's exponential and embedding $\Bbb Z\to \Bbb C_\infty$ , if $m\in\Bbb Z$ then $Nv^m$ has the usual meaning. (See loc.cit.) Since we are only interested in the special values at $s\in\Bbb Z_{\geq 0}$ of this $L$--function, where the infinite product above is convergent by the arguments below, we leave general convergence and well-definedness issues asside, and will treat them in a separatate, upcoming paper.

\begin{remark} Note that in \cite{FGHP20} (see \S1.2)  we worked with a slightly different definition of the polynomial $P_v^{\ast, G, E}(X)$ (denoted in loc.cit. $P_v^{\ast, G}(X)$.) We adopted the new definition given above, to be consistent with Goss's definition of (non--equivariant) $L$--functions for abelian $t$--motives given in \S8.6 of \cite{Goss}. However, note that the value $P_v^{\ast, G, E}(1)$, which will be of interest to us in what follows,  does not change under the new definition.
\end{remark}

Although the special value
$\Theta_{K/F, S}^E(0)$ is well defined, for technical reasons (as explained in loc.cit.) we focus on computing certain Euler--completed versions of it $\Theta_{K/F}^{E,\mathcal M}(0)$, which depend on additional arithmetic data $\mathcal M$. For well chosen data $\mathcal M$, one has an equality
$$\Theta_{K/F, S}^E(0)=\Theta_{K/F}^{E, \mathcal M}(0),$$
so nothing is lost in the process of Euler--completion. (See \S\ref{incomplete-L-section} for details.) A detailed study of these Euler--incomplete  $L$--functions and a proof of the formulas for their special values will be addressed in a separate paper. In this section, we only define the special values $\Theta_{K/F}^{E,\mathcal M}(0)$, whose study will be the main focus of this work.\\

An important ingredient in defining the special values above is the notion of monicity in the ring $\F_q[G]((t^{-1}))$, as defined in \cite[\S 7.3]{FGHP20}.  This generalizes the classical notion of monicity in $\Bbb F_q((t^{-1}))$. We state the definitions and results here, but refer the reader to loc.cit. for  details.  First, we decompose $G = P\times \Delta$, where $P$ is the $p$-Sylow subgroup of $G$. Second, we let $\widehat\Delta(\Bbb F_q)$ be the set of $G_{\Bbb F_q}$--conjugacy classes of irreducible, $\overline{\Bbb F_q}$--valued characters of $\Delta$.
\begin{definition} For an algebraic extension $\Bbb F/\Bbb F_q$, define the subgroup $\F((t^{-1}))[P]^+$  of monic elements in $\F((t^{-1}))[P]^\times$ by
$$\F((t^{-1}))[P]^+:=\bigcup_{n\in\Z}\, t^n\cdot(1+t^{-1}\F[P][[t^{-1}]]).$$
\end{definition}
\noindent Choosing representatives $\chi$ for the conjugacy classes $\widehat\chi\in\widehat\Delta(\Bbb F_q)$ gives a ring isomorphism
\[\psi_\Delta: \F_q[G]((t^{-1}))\simeq \bigoplus_{\widehat\chi}\F_q(\chi)[P]((t^{-1})).\]

\begin{definition}
Define the subgroup $\F_q((t^{-1}))[G]^+$  of monic elements in $\F_q((t^{-1}))[G]^\times$ by
$$\F_q((t^{-1}))[G]^+:=\psi_{\Delta}^{-1}\left(\bigoplus_{\widehat\chi}\F_q(\chi)((t^{-1}))[P]^+\right).$$
\end{definition}

The following results are proved in \cite[\S\S7.3.4-7.3.6]{FGHP20}.

\begin{proposition}\label{P:MonicProp} The following hold true.
\begin{enumerate}
\item There is an equality of groups
\[\F_q((t^{-1}))[G]^\times=\F_q((t^{-1}))[G]^+\times \F_q[t][G]^\times.\]
\item There is a canonical group isomorphism
\[\F_q((t^{-1}))[G]^\times/\F_q[t][G]^\times \simeq \F_q((t^{-1}))[G]^+, \qquad \widehat g \to g^+,\]
sending the class $\widehat x$ of $x\in\F_q((t^{-1}))[G]^\times$ to its unique monic representative $x^+.$
\item The monoid $\Bbb F_q[t][G]^+:=\Bbb F_q[t][G]\cap \F_q((t^{-1}))[G]^+$ consists of polynomials $f\in\Bbb F_q[G][t]$, such that
$\chi(f)$ is a (classically) monic polynomial in $\Bbb F_q(\chi)[P][t]$, for all $\chi\in\widehat\Delta$.
\end{enumerate}
\end{proposition}

\begin{proposition}
Assume that $B$ is a finite $A[G]$--module which is $\F_q[G]$--projective. Then:
\begin{enumerate}
\item The Fitting ideal $\Fitt_{A[G]}^0(B)$ is principal and has a unique monic generator
$$f_B(t)\in\Bbb F_q[G][t]^+.$$
\item If $B$ is free of rank $m$ as an $\Bbb F_q[G]$--module, then $f_B(t)$ is a degree $m$ monic polynomial, in the classical sense.
\end{enumerate}
\end{proposition}

\begin{definition} For $B$ as in the above Proposition, define
\[\lvert B\rvert_G := f_B(t) \in \F_q[G][t]^+.\]
and call it the $A[G]$--size (or $G$--size) of $B$.
\end{definition}

The following Proposition provides the examples of $A[G]$--modules $B$ as above which are directly involved in defining the special $L$--values $\Theta_{K/F}^{E,\mathcal M}(0)$ of interest to us.
\begin{proposition}
For the set of data $(K/F, E)$ as above, the following hold true.
\begin{enumerate}
\item If $v\in \MSpec(\cO_F)$ is a prime which is tamely ramified in $K/F$, then $\Lie_E(\cO_K/v)$ and $E(\cO_K/v)$ are free $\F_q[G]$-modules of rank $n\cdot[\cO_F/v:\F_q]$.
\item There exists an $\cO_F[G]\{\tau\}$--submodule $\cM$ of $\cO_K$, called a taming module for $\mathcal O_K/\mathcal O_F$, or simply a taming module, satisfying the following properties.
\begin{enumerate}
\item $\cM$ is a projective $\cO_F[G]$--module.
\item The quotient $\cO_K/\cM$ is finite and supported only at primes $v\in{\rm MSpec}(\cO_F)$ which are wildly ramified in $K/F$.
\end{enumerate}
\item For any taming module $\mathcal M$, the $A[G]$--modules $\Lie_E(\cM/v)$ and $E(\cM/v)$ are $\F_q[G]$--free of rank $n\cdot[\cO_F/v:\F_q]$, for all
$v\in{\rm MSpec}(\cO_F)$.
\end{enumerate}
\end{proposition}

\begin{proof}
Observe that there are isomorphisms of $\Bbb F_q[G]$--modules (not of $A[G]$--modules!)
$$\Lie_E(\cO_K/v)\simeq E(\cO_K/v)\simeq (\cO_K/v)^n, \qquad \Lie_E(\cM/v)\simeq E(\cM/v)\simeq (\cM/v)^n, $$
for all $v\in{\rm MSpec}(\mathcal O_F)$ and all taming modules $\mathcal M$. Now, apply \cite[Prop 1.2.5, Prop. 7.2.4]{FGHP20}.
\end{proof}

Now, we fix a taming module $\mathcal M$. The Proposition above permits us to consider the monic polynomials $|\Lie_E(\cM/v)|_G$ and $|E(\cM/v)|_G$ in $\Bbb F_q[G][t]$,  whose degrees are both equal
to $n\cdot[\cO_F/v:\F_q]$, for all $v\in{\rm MSpec}(\cO_F)$. As a consequence, we may consider the quotients
$$\frac{|\Lie_E(\cM/v)|_G}{|E(\cM/v)|_G}\in (1+t^{-1}\Bbb F_q[G][[t^{-1}]]).$$

\begin{definition}\label{L-value-definition}
We define the $L$-value at $s=0$ associated to the data $(K/F,E,\cM)$ by
\[\Theta_{K/F}^{E,\cM}(0) := \prod_{v\in{\rm MSpec}(\cO_F) }\frac{|\Lie_E(\cM/v)|_G}{|E(\cM/v)|_G}\quad \in (1+t^{-1}\Bbb F_q[G][[t^{-1}]]).\]
\end{definition}
\noindent The fact that the infinite Euler product above converges (in the $t^{-1}$--adic topology) to an element in $(1+t^{-1}\Bbb F_q[G][[t^{-1}]])$
will be a direct consequence of our results below.
\begin{remark}
Note that for a finite set $S\subseteq {\rm MSpec}(\mathcal O_F)$ which contains the wild ramification locus $W$ for $\mathcal O_K/\mathcal O_F$, one can define the $S$--incomplete infinite Euler product 
\[\Theta_{K/F, S}^{E}(0) := \prod_{v\in{\rm MSpec}(\cO_F)\setminus S }\frac{|\Lie_E(\cO_K/v)|_G}{|E(\cO_K/v)|_G}\quad \in (1+t^{-1}\Bbb F_q[G][[t^{-1}]]).\]
It is shown in \S\ref{incomplete-L-section} below that for any such $S$ (in particular, for $S=S_0$ or $S=W$) one can construct taming modules $\cM$ (called $(E, S, \cO_K/\cO_F)$--taming in \S\ref{incomplete-L-section}), for which
 \[\Theta_{K/F, S}^E(0)= \Theta_{K/F}^{E,\cM}(0).\]
 So, understanding the Euler--completed special values in Definition \ref{L-value-definition} suffices.
\end{remark}

\begin{remark}\label{actual-values-remark}
It is interesting to note that, unlike the $L$--functions themselves,  the special $L$--values above make sense even if the $t$--module $E$ is not abelian. It is also the case that for our study of these special values, which is the main object
of this paper, we are not using the abelianness hypothesis on $E$.

However, for abelian $t$--modules $E$, and $v\in{\rm MSpec}(\mathcal O_F)\setminus S_0$, it is expected to have 
\begin{equation}\label{Euler-factors} P_v^{\ast, G, E}(1)^{-1}=\frac{|\Lie_E(\mathcal O_K/v)|_G}{|E(\mathcal O_K/v)|_G},\end{equation}
and therefore the infinite Euler product introduced in the last Remark above is expected to be the actual value at $s=0$ of the Euler--incomplete $L$--function
$\Theta_{K/F, S}^E(s)$ defined in \eqref{incomplete-L-definition} above, for all finite sets $S$ containing $S_0$. Note that equality \eqref{Euler-factors} above was proved for all Drinfeld modules $E$ in \cite{PR24}
and a strategy on how to prove it for pure abelian $t$--modules $E$ was also provided in the last section of loc.cit.  
\end{remark}

\subsection{Lattices, lattice--indices, and volumes} In this section, we extend the definitions, constructions and results of \S4 of \cite{FGHP20} to the context of $t$--modules.
Note that $\Lie_E(\cO_K)$ sits inside $\Lie_E(K_\infty)$ diagonally as a discrete, cocompact $A[G]$-submodule (because $\mathcal O_K$ sits inside $K_\infty$ that way, see loc.cit.) Also, recall that
$\Lie_E(K_\infty)$ is a free $k_\infty[G]$--module of rank $nr$, where $r:=[F:k]$ (see Lemma \ref{kinfty[G]-free-Lemma}.)

\begin{definition} (Lattices in $\Lie_E(K_\infty)$, see \cite[\S 4]{FGHP20}.)
\begin{enumerate}
\item An $A$-lattice in $\Lie_E(K_\infty)$ is a free $A$-submodule of $\Lie_E(K_\infty)$ of rank equal to ${\rm dim}_{k_\infty}\Lie_E(K_\infty)$, which spans $\Lie_E(K_\infty)$ as a $k_\infty$-vector space.
\item An $A[G]$-lattice in $\Lie_E(K_\infty)$ is an $A[G]$-submodule of $\Lie_E(K_\infty)$ which is an $A$-lattice in $\Lie_E(K_\infty)$.
\item A projective (respectively, free) $A[G]$-lattice in $\Lie_E(K_\infty)$ is an $A[G]$-lattice  which is projective (respectively, free) as an $A[G]$-module.
\end{enumerate}

\end{definition}

\begin{proposition}
If $\cM$ is a taming module for $\mathcal O_K/\mathcal O_F$, then $\Lie_E(\mathcal M)$ and $\Exp_E\inv(E(\cM))$ are $A[G]$-lattices in $\Lie_E(K_\infty)$.
\end{proposition}
\begin{proof}
It is proved in \cite[Prop. 2.7-2.8]{Dem} that $\Lie_E(\cO_K)$ and $\Exp_E\inv(E(\cO_K))$ are $A$-lattices and it is clear that they are also $G$-modules.  The result then follows since $\cO_K/\cM$ is finite and since $\cM$ is also a $G$-module.
\end{proof}

Let $M$ be either a taming module $\cM$ or $\cO_K$. Then, we have the following exact sequence of topological $A[G]$-modules,
\begin{equation}
0  \longrightarrow  \Lie_E(K_\infty)/\Exp_E\inv(M)  \overset{\Exp_E}{\longrightarrow}  E(K_\infty)/E(M)  \overset{\pi}{\longrightarrow}   H(E/M)  \longrightarrow 0,
\end{equation}
where $H(E/M)$ is defined to be the cokernel of the exponential map,
\begin{equation}
H(E/M) := \frac{E(K_\infty)}{E(M) + \Exp_E(K_\infty)}.
\end{equation}
It is shown in \cite[Prop. 2.8]{Dem} that $H(E/\cO_K)$ is finite.  Now, from the definitions,  we have an exact sequence of $A[G]$--modules
\begin{equation}\label{hm-ho-equation}\frac{E(\cO_K)}{E(\cM) + \Exp_E(\Exp_E\inv(E(\cO_K)))}\to H(E/\cM)\to H(E/\cO_K)\to 0.\end{equation}
Since $E(\cO_K)/E(\cM)$ is finite, the finiteness of $H(E/\cO_K)$ implies that of $H(E/\cM)$, for all taming modules $\cM$.
The modules $H(E/\cO_K)$ and $H(E/\cM)$ are natural $t$--module generalizations of Taelman's class--modules, defined in the context of Drinfeld modules in \cite{T12}.

\begin{definition} (Lattice index, see \cite[\S 4.1]{FGHP20}.)
\begin{enumerate}
\item If $\Lambda_1,\Lambda_2 \subseteq \Lie_E(K_\infty)$ are free $A[G]$-lattices, let $X \in \GL_{nr}(k_\infty[G])$ be a change of basis matrix between them. Then we define the index
\[[\Lambda_1:\Lambda_2]_G := \det(X)^+.\]
We note that while $X$ is not unique, the monic representative of its determinant $\det(X)^+$ is unique. (See loc.cit.)
\item If $\Lambda_1,\Lambda_2 \subset \Lie_E(K_\infty)$ are projective $A[G]$-lattices, choose free $A[G]$-lattices $\cF_1 \supseteq \Lambda_1$ and $\cF_2 \supseteq \Lambda_2$ (mimic the proof in loc.cit. for existence)
and define
\[[\Lambda_1:\Lambda_2]_G := [\cF_1:\cF_2]_G \cdot \frac{|\cF_2/\Lambda_2|_G}{|\cF_1/\Lambda_1|_G}.\]
This definition is independent of the choice of $\cF_1$ and $\cF_2$. (See loc.cit.)
\end{enumerate}
\end{definition}

\begin{definition}\label{D:classC} (The Arakelov class, see \cite[\S 4.2]{FGHP20}.)
Let $\cC$ be the class of compact $A[G]$-modules $M$ which are $G$--c.t. and fit in a short
exact sequence of topological $A[G]$--modules
\[0\longrightarrow \Lie_E(K_\infty)/\Lambda \overset{\iota}\longrightarrow M \overset{\pi}\longrightarrow H\longrightarrow 0,\]
where $\Lambda$ is an $A[G]$--lattice in $\Lie_E(K_\infty)$ and $H$ is a finite $A[G]$--module.
\end{definition}

Note that $\Lie_E(K_\infty)$ is $A$-divisible, and thus $A$--injective. Thus, in the category of $A$-modules the above exact sequence splits. We let $s:H\to M$ be an $A$--linear  splitting map for $\pi$. Thus we have an $A$-module isomorphism
\[\Lie_E(K_\infty)/\Lambda \times s(H) \isom M.\]
It is important to note that, in general, the sequence above does not split in the category of $A[G]$-modules; this depends on the $G$--cohomology of $M$. 

Now, we follow \cite[4.2.2-4.2.3]{FGHP20} to introduce the notion of an $A[G]$-admissible lattice for an element $M\in\mathcal C$  and to  show that such lattices always exist.

\begin{definition}\label{D:Admissible lattice} For $M$ an element of the class $\cC$ and a section $s:H\to M$ as above, an  $A[G]$--lattice $\Lambda'$ in $\Lie_E(K_\infty)$ is called $(M, s)$--admissible if
\begin{enumerate}
\item $\Lambda\subseteq\Lambda'$;
\item $\Lambda'$ is $A[G]$--projective;
\item $\Lambda'/\Lambda\times s(H)$ is an $A[G]$--submodule of $M$.
\end{enumerate}
An $A[G]$--lattice $\Lambda'$ is called $M$--admissible if it is $(M,s)$--admissible for some $s$ as above.
\end{definition}

\begin{proposition}\label{P:Admissible}
For $(M, s)$ as above, there exist $A[G]$--free, $(M,s)$--admissible lattices. Further, for such a lattice $\Lambda'$, we have a short exact sequence of $A[G]$-modules
\begin{equation}\label{E:admissible-sequence} 0 \rightarrow \Lambda'/\Lambda\times s(H)\rightarrow M \rightarrow K_\infty/\Lambda'\rightarrow 0.
\end{equation}
Consequently $\Lambda'/\Lambda\times s(H)$ is a finite, $G$--c.t. $A[G]$-module.
\end{proposition}
\begin{proof} Mimic the proofs in loc.cit. and keep in mind that $\Lie_E(K_\infty)$ is $k_\infty[G]$--free.
\end{proof}

\begin{definition}[Volume Definition]
We first fix a (normalizing) projective $A[G]$-lattice $\Lambda_0\subset K_\infty^n$. For $M\in \cC$, let $s$ be a splitting map for $M$ and let $\Lambda'$ be an $(M,s)$-admissible lattice as above. Define
\begin{equation}\label{D:Vol}
\Vol(M) = \frac{|\Lambda'/\Lambda \times s(H)|_G}{\left [\Lambda':\Lambda_0\right ]_{G}},
\end{equation}
\end{definition}

We now give a collection of facts regarding the volume function, following \cite{FGHP20}.

\begin{proposition}[see Proposition 4.2.8 of \cite{FGHP20}]
The function $\Vol:\cC\to \laurent{\F_q}{T\inv}[G]^+$ satisfies the following properties.
\begin{enumerate}
\item  For each $M \in \cC$ given by an exact sequence as in Definition \ref{D:classC}, the value $\Vol(M)$ is independent of  choice of section $s$ and of choice of
$(M,s)$--admissible lattice $\Lambda'$.
\item  $\Vol(\Lie_E(K_\infty)/\Lambda_0) = 1$.
\item  If $M_1,M_2 \in \cC$, the quantity $\frac{\Vol(M_1)}{\Vol(M_2)}$ is independent of choice of $\Lambda_0$.
\item If $M\in\cC$, then ${\rm Vol}(M)$ depends only on the extension class $[M]\in{\rm Ext}^1_{A[G]}(H, K_\infty/\Lambda)$ (if $H$ and $\Lambda$ are fixed.)
\end{enumerate}
\end{proposition}

\begin{proof} See the proof of Proposition 4.2.8 in \cite{FGHP20}.
\end{proof}

\subsection{Main results} Now, we are ready to state the main theorems of this paper, proved in \S\ref{main-results-section} below. The data $K/F/k$ and $\cO_K/\cO_F/A$ are as in \S1.2 and $E$ is an $n$--dimensional $t$--module of structural morphism 
$\phi_E:A\to M_n(\cO_F)\{\tau\}.$\\

The main theorem, whose proof occupies most of this paper, is the following Equivariant Tamagawa Number Formula for $t$--modules. This extends  our main result in \cite{FGHP20} from the Drinfeld module to the $t$--module setting and is a $G$--equivariant refinement and $t$--module generalization of Taelman's original class--number formula for Drinfeld modules \cite{T12}.
\begin{theorem}[the ETNF for $t$-modules]\label{T:ETNF}  If $\cM$ is a taming module for $\cO_K/\cO_F$, then we have the following equality in $(1+t^{-1}\F_q[[t^{-1}]][G])$.
$$\Theta_{K/F}^{E, \cM}(0)=\frac{{\rm Vol}(E(K_\infty)/E(\cM))}{{\rm Vol}(\Lie_E(K_\infty)/\Lie_E(\cM))}.$$
\end{theorem}

\begin{remark} A result similar to the theorem above, but under certain restrictive conditions, was also obtained by T. Beaumont in \cite{Beaumont}. 
\end{remark}
\medskip

An important  corollary to the above Theorem is the following $t$--module analogue of the classical (number field) Refined Brumer--Stark Conjecture. (See detailed comments on the analogy in \S\ref{BrSt-section} and \S\ref{positive-integers-section} below.)

\begin{theorem}[Refined Brumer--Stark for $t$--modules]\label{T:BrSt}  Let $\cM$ be taming module for $\cO_K/\cO_F$ and let $\Lambda'$ be a $E(K_\infty)/E(\cM)$--admissible $A[G]$--lattice  in $\Lie_E(K_\infty)$. Then, we have
$$\frac{1}{[\Lie_E(\cM):\Lambda']_G}\cdot \Theta_{K/F}^{E, \cM}(0)\in{\rm Fitt}^0_{A[G]}H(E/\cM)\subseteq {\rm Fitt}^0_{A[G]}H(E/\cO_K).$$
\end{theorem}

In order to study the imprimitive (Euler incomplete) special values $\Theta_{K/F, S}^E(0)$ discussed in \S\ref{L-values-section} above, one needs taming modules $\mathcal M$ which satisfy some additional conditions. These are called $(E, S, \cO_K/\cO_F)$--taming. 
See \S\ref{incomplete-L-section} below for the precise definition and existence. The result we obtain is the following.
\begin{theorem}[imprimitive ETNF]
Let $E$ be an abelian $t$-module as above. Let $S$ be a finite subset of ${\rm MSpec}(\mathcal O_F)$ which contains the wild ramification locus $W$ for $\cO_K/\cO_F$. Let  $\cM$ be any $(E, S, \cO_K/\cO_F)$--taming module. Then, we have the following equality in $(1+t^{-1}\F_q[[t^{-1}]][G])$:
\[\Theta_{K/F, S}^{E}(0)=\frac{{\rm Vol}(E(K_\infty)/E(\cM))}{{\rm Vol}(\Lie_E(K_\infty)/\Lie_E(\cM))}.\]
\end{theorem}
\medskip

If $E$ is a Drinfeld module, one can twist it with powers of the Carlitz module $\mathcal C$ and obtain the $(1+m\cdot{\rm rank}(E))$--dimensional $t$--modules $E(m):=E\otimes\mathcal C^{\otimes m}$. By using the so--called universally taming modules for the 
data $(S, \cO_K/\cO_F)$ (see \S\ref{incomplete-L-section} for the definition and existence), we obtain the following equivariant Tamagawa number formula for the values $\Theta_{K/F, S}^E(m)$, at all $m\in\Bbb Z_{\geq 0}$.
\begin{theorem}[ETNF at positive integers] Let $E$ be a Drinfeld module of structural morphism 
$\phi_E:A\to\mathcal O_F\{\tau\}$. Let $S$ be a finite set of primes in ${\rm MSpec}(\mathcal O_F)$, containing the primes of bad reduction for $E$ and the wildly ramified primes in $\mathcal O_K/\mathcal O_F$.
Let $\cM$ be a $(\cO_K/\cO_F, S)$--universally taming module. Then, we have the following equalities in $(1+t^{-1}\F_q[[t^{-1}]][G])$
\[\Theta_{K/F, S}^{E}(m)=\frac{{\rm Vol}(E(m)(K_\infty)/E(m)(\cM))}{{\rm Vol}(\Lie_{E(m)}(K_\infty)/\Lie_{E(m)}(\cM))},\] 
for all $m\in\Bbb Z_{\geq 0}.$
\end{theorem}

A notable consequence of the above theorem is the following Drinfeld module analogue of the classical (number field) Refined Coates--Sinnott Conjecture, linking special $L$--values at negative integers to the higher Quillen $K$--groups of the ring of algebraic integers in a number field.
 (See \S\ref{positive-integers-section} for more detailed comments on the analogy.)

\begin{theorem}[Refined Coates--Sinnott for Drinfeld modules]\label{T:CS}  For data as in the last theorem,  let $\Lambda'$ be a $E(m)(K_\infty)/E(m)(\cM)$--admissible $A[G]$--lattice  in $\Lie_{E(m)}(K_\infty)$. Then
$$\frac{1}{[\Lie_{E(m)}(\cM):\Lambda']_G}\cdot \Theta_{K/F, S}^{E}(m)\in{\rm Fitt}^0_{A[G]}H(E(m)/\cM)\subseteq {\rm Fitt}^0_{A[G]}H(E(m)/\cO_K),$$
for all $m\in\Bbb Z_{\geq 0}$.
\end{theorem}
As noted in the abstract, Theorems \ref{T:BrSt} and \ref{T:CS} are key ingredients in developing an Iwasawa theory for Taelman's class modules. This will be addressed in an upcoming paper.\medskip

\centerline{--------------}
\medskip

The paper is organized as follows. In \S\ref{S:Nuclear}, we extend to the $t$--module context the $G$--equivariant theory of nuclear operators developed in \cite{FGHP20} for Drinfeld modules. In the same section, we prove a $G$--equivariant trace formula for certain nuclear operators in the $t$--module context. 
In \S\ref{S:Volume}, we use this trace formula to prove a volume formula for objects in the Arakelov class $\mathcal C$. Finally, in \S\ref{main-results-section}, we use this volume formula to prove the main theorem (Theorem \ref{T:ETNF}) and derive all its consequences stated above. The paper relies heavily on  techniques and results developed in \cite{FGHP20}, so the readers are strongly advised to familiarize 
themselves with that paper before embarking on this journey.

\section{Nuclear Operators and the Trace Formula}\label{S:Nuclear}
We briefly review some of the theory on  nuclear operators in the equivariant setting, developed in \S2 of \cite{FGHP20}.  Most of the results were proved in sufficient generality in \cite{FGHP20}, so we may use them here without alteration.  When minor alterations are required, we note the differences in the proofs.  Only in cases where substantial differences exist do we give full proofs in this section.
Let $R := \F_q[G]$ and let $V$ be a compact, projective, topological $R$-module.

\begin{definition}\label{D:OpenNeighborhood}
We let $\cU = \{U_i\}_{i\geq M}$, for some $M\geq 0$, be a sequence of open $R$-submodules of $V$ such that:
\begin{enumerate}
\item Each $U_i$ is $G$-c.t.
\item $U_{i+1}\subset U_i$ for all $i\geq M$.
\item $\cU$ is a basis of open neighborhoods of $0$ in $V$.
\end{enumerate}
\end{definition}
\begin{remark}As in loc.cit. if $V$ is a finitely generated $R[G]$--module (and therefore finite), we always take $U_i=\{0\}$, for all $i\in\Bbb Z_{\geq 0}$.
\end{remark}
Let $Z$ be a variable. Define the $R[[Z]]/Z^N$-module (respectively, $R[[Z]]$--module)
\[V[[Z]]/Z^N := V\otimes_R R[[Z]]/Z^N,  \text{ and } V[[Z]] := \varprojlim_{N\geq 1} V[[Z]]/Z^N,\]
and endow these with the natural (direct sum, respectively inverse limit) topologies.  We note that any continuous $R[[Z]]/Z^N$-linear (respectively, $R[[Z]]$--linear) endomorphism $\Phi$ of $V[[Z]]/Z^N$ (respectively, of $V[[Z]]$) is of the form
\[\Phi = \sum_{n=0}^{N-1} \phi_n Z^n,\quad \text{(respectively, $\Phi = \sum_{n=0}^\infty \phi_n Z^n$)},\]
where the $\phi_n$'s are continuous $R$--linear endomorphisms of $V$, for all $n$.

\begin{definition}[locally contracting and nuclear endomorphisms]${}$
\begin{enumerate}
\item We say that a continuous endomorphism $\phi$ of $V$ is locally contracting if there exists $M'\geq M$ such that $\phi(U_i)\subset U_{i+1}$ for all $i\geq M'$.  We will call $U_{M'}$ a nucleus for $\phi$.
\item We call a continuous $R[[Z]]$-linear endomorphism $\Phi$ of $V[[Z]]$ nuclear if for all $n\geq 0$ the endomorphisms $\phi_n$ of $V$ are locally contracting.  We define nuclear endomorphisms for $V[[Z]]/Z^N$ similarly.
\end{enumerate}
\end{definition}

Next, we state several facts about locally contracting and nuclear endomorphisms and the ensuing nuclear determinants. The reader should follow closely the proofs in \cite[\S 2.1]{FGHP20}.
\begin{proposition}[properties of endomorphisms]${}$
\begin{enumerate}
\item Any finite collection of locally contracting endomorphisms of $V$ has a common nucleus.
\item If $\phi$ and $\psi$ are locally contracting endomorphisms of $V$, then so are the sum $\phi+\psi$ and the composition $\phi\psi$.
\item Let $\Phi:V[[Z]]/Z^N\rightarrow V[[Z]]/Z^N$ be a nuclear endomorphism. Let $U$ and $W$ be common nuclei for the endomorphisms $\phi_n$'s. Then
\[\det\nolimits_{R[[Z]]/Z^N}(1 + \Phi\rvert V/U) = \det\nolimits_{R[[Z]]/Z^N}(1 + \Phi\rvert V/W).\]
\end{enumerate}
\end{proposition}

\begin{definition}[nuclear determinants]${}$
\begin{enumerate}
\item For $\Phi$ a nuclear endomorphism of $V[[Z]]/Z^N$, we define the determinant of $(1+\Phi)$
\[\det\nolimits_{R[[Z]]/Z^N}(1 + \Phi\rvert V) := \det\nolimits_{R[[Z]]/Z^N}(1 + \Phi\rvert V/U),\]
where $U$ is any common nucleus for the endomorphisms $\phi_n$.
\item For $\Phi$ a nuclear endomorphism of $V[[Z]]$, we define the determinant of $(1+\Phi)$ in $\ds R[[Z]] = \varprojlim_N R[[Z]]/Z^N$ by
\[\det\nolimits_{R[[Z]]}(1 + \Phi\rvert V):=\varprojlim_N\det\nolimits_{R[[Z]]/Z^N}(1 + \Phi\rvert V).\]
\end{enumerate}
\end{definition}

\begin{proposition}[properties of nuclear determinants]\label{P:DeterminantFacts}${}$
\begin{enumerate}
\item (Multiplicativity of Determinants) If $\Phi$ and $\Psi$ are nuclear endomorphisms of $V[[Z]]$, then  the endomorphism $(1 + \Phi)(1 + \Psi) - 1$ is nuclear, and
\[\det\nolimits_{R[[Z]]}((1 + \Phi)(1 + \Psi)\rvert V) = \det\nolimits_{R[[Z]]}(1 + \Phi\rvert V)\det\nolimits_{R[[Z]]}(1 + \Psi\rvert V).\]
\item (Exact Sequences) Let $V'\subseteq V$ be a closed $R$-submodule of $V$ which is $G$--c.t. and let $V'' := V/V'$. Let $\mathcal{U}' = \{U_i'\}_i$ where $U_i' = U_i\cap V'$, and $\mathcal{U}'' = \{U_i''\}_i$ where $U_i''$ is the image of $U_i$ in $V''$. Assume that all the $U_i'$ and $U_i''$ are $G$-c.t. Let $\Phi = \sum \phi_nZ^n:V[[Z]]\rightarrow V[[Z]]$ be a nuclear endomorphism, such that $\phi_n(V')\subseteq V'$, for all $n$. Then the endomorphisms induced by $\Phi$ on $(V', \mathcal{U}')$ and $(V'',\mathcal{U}'')$ are nuclear and
$$\det\nolimits_{R[[Z]]}(1 + \Phi\rvert V) = \det\nolimits_{R[[Z]]}(1 + \Phi\rvert V')\det\nolimits_{R[[Z]]}(1 + \Phi\rvert V'').$$
\item (Independence of Basis of Open Neighborhoods) Let $\cU =  \{U_i\}$ and $\cU' =  \{U'_i\}$ be bases of open neighborhoods for $V$, satisfying the properties in Definition \ref{D:OpenNeighborhood}.  Let $\Phi$ be an endomorphism of $V[[Z]]$ which is nuclear with respect to both $\cU$ or $\cU'$.  Assume that for all $n\geq 0$, there exists an $M(n)\in\Z_{\geq 0}$ such that
for all $i\geq M(n)$ there exists a $j(i)\geq M(n)$ with the property that
$$U_{i}\supseteq U'_{j(i)} \text{ and } \phi_n(U_{i})\subseteq U'_{j(i)}.$$
Then, we have an equality
$${\rm det}_{R[[Z]]}(1+\Phi|V)={\rm det}{'}_{R[[Z]]}(1+\Phi|V),$$
where the determinant on the left is calculated with respect to $\cU$ and the determinant on the right with respect to $\cU'$.
\end{enumerate}
\end{proposition}

\begin{remark}\label{R:finite-det}
Assume that $V$ is a finite $A[G]$--module (i.e. an $R[t]$--module) which is $R$--free of rank $m$. Then we can view $\Phi:=-t\cdot T^{-1}$ as a nuclear endomorphism of $V[[T^{-1}]]$. Then $\det_{R[[T^{-1}]]}(1-t\cdot T^{-1}\mid V)$ as defined above is the usual determinant of $(1+\Phi)$ viewed as an endomorphism of the free $R[[T^{-1}]]$--module $V\otimes_RR[[T^{-1}]]$ of rank $m$.
We have the following equality in $R[t]$ (see Remark 2.1.11 in \cite{FGHP20}):
$$|V|_G= t^m\cdot\det\nolimits_{R[[T^{-1}]]}(1-t\cdot T^{-1}\mid V)\rvert_{T=t}.$$
\end{remark}

\medskip
Now, we describe the modules $V$ and bases of open neighborhoods $\cU$ which we will use throughout the rest of this paper.  Let $v$ be a prime of $F$ and let $K_v = \prod_{w|v} K_w$ be the product of the $w$-adic completions of $K$ for all primes $w$ of $K$, sitting above $v$.  Let $F_v$, $\cO_v$ and $\fm_v$ be the $v$-adic completion of $F$, its ring of integers and its maximal ideal, respectively.  Let $S_\infty$ denote the set of infinite primes of $F$ and set $K_\infty = \prod_{v\in S_\infty} K_v$.  Let $\cO_{F,\infty}$ be the intersection of the valuations rings in $F$ corresponding to the places of $S_\infty$ and let $\cO_{K,\infty}$ be its integral closure in $K$. (Note that these are both semi-local PIDs.)

\begin{definition}
Let $\cR$ be a Dedekind domain of field of fractions $F$ and let $\cS$ be its integral closure in $K$.  We say that an $\cR\{\tau\}[G]$-submodule $\cM\subset \cS$ is a taming module for $\cS/\cR$ if
\begin{enumerate}
\item $\cM$ is $\cR[G]$--projective of constant local rank $1$.
\item $\cS/\cM$ is finite and $(\cS/\cM)\otimes_{\cR}\cO_v=0$ whenever $v\in\mspec(\cR)$ is tame in $K/F$.
\end{enumerate}
\end{definition}

\begin{proposition}[see \S7.2 in \cite{FGHP20}]\label{P:taming-module} For $\cR$ and $\cS$ as in the definition above, the following hold.
\begin{enumerate}
\item Taming modules $\cM$ for $\cS/\cR$ exist.
\item If $\cS/\cR$ is tame, then any such $\cM$ equals $\cS$.
\item For any such $\cM$ and any $v\in\mspec(\cR)$, we have an $\Fq[G]$--module isomorphism
$$\cM/v\simeq \Fq[G]^{n_v}, \qquad \text{ where } n_v:=[\cR/v:\Fq].$$
\item For any such $\cM$ and $v\in\mspec(\cR)$ which is tame in $K/F$, we have $\cM/v=\cS/v$.
\item For $v\in\mspec(\cR)$, let $\cM_v$ be the $v$--adic completion of $\cM$ and let $\pi_v\in\cR$, such that $v(\pi_v)>0$. Then $\left\{\pi_v^i\cM_v\right\}_{i\geq 0}$ is a basis of open neighborhoods of $0$ in $K_v$ consisting of free $\cO_v[G]$--modules of rank $1$.
\end{enumerate}
\end{proposition}

\begin{definition}
We call a taming module for $\cO_K/\cO_F$ simply a taming module for $K/F$ and call a taming module for $\cO_{K,\infty}/\cO_{F,\infty}$ an $\infty$-taming module for $K/F$.
\end{definition}

\begin{definition}\label{D:UNeighborhoods}
Now let $(\cM, \cM^\infty)$ be a taming pair for $K/F$ consisting of a taming module $\cM$ and an $\infty$-taming module $\cM^\infty$ for $K/F$.  Below, we define bases of open neighborhoods satisfying Definition \ref{D:OpenNeighborhood} for three different types of compact, projective $R$-modules V.
\begin{enumerate}
\item For a prime $v$ of $F$, let $V := K_v^n$. In this case, define
\[\{\bU_{i, v}\}_{i\geq 0}:=\{(t^{-i}\mathcal M_v^\infty)^n\}_{i\geq 0}, \qquad\qquad \{\bU_{i,v}\}_{i\geq 0}:=\{(\frak m_v^i\mathcal M_v)^n\}_{i\geq 0},\]
for $v\in S_\infty$ and $v\notin S_\infty$, respectively, where $\cM_v$ and $\cM_v^\infty$ denote the obvious $v$-adic completions.  Note that these neighborhoods satisfy the properties in Definition \ref{D:OpenNeighborhood}, since finite products of $G$-c.t. (therefore projective) $R$--modules
are also $G$-c.t.
\item Let $V$ be an $A[G]$-module in class $\cC$ (Definition \ref{D:classC}), of structural exact sequence
\[0\longrightarrow \Lie_E(K_\infty)/\Lambda \overset{\iota}\longrightarrow V\longrightarrow H\longrightarrow 0.\]
In particular, $V$ is a compact, projective $R$--module.
We let $\bU_{i,\infty} = \prod_{v\in S_\infty} \bU_{i,v}$ and recall that $\Lambda$ is discrete in $K_\infty^n$.  Thus there exists $\ell \geq 0$ such that $\bU_{i,\infty}\cap \Lambda = \{0\}$, for all $i\geq \ell$.  Thus, we define a basis of open neighborhoods for $V$ by setting
\[\cU = \{\iota(\bU_{i,\infty})\}_{i\geq \ell}.\]
For simplicity of notation, we will often omit $\iota$ and think of it as simple inclusion.
\item Now, let $S$ be a finite set of primes in $F$ containing $S_\infty$ and set $K_S = \prod_{v\in S} K_v$. Let
$$\OX_{F,S} = \{\alpha\in F: v(\alpha) \geq 0, \text{ for } v\not\in S\}$$
 be the ring of $S$-integers in $F$. Let $\mathcal{M}_S = \mathcal{M}\otimes_{\OX_F}\OX_{F,S}\subseteq K_S$ and let
 $$V:=(K_S/\cM_S)^n.$$
As above, set
\[\bU_{i,S} = \prod_{v\in S} \bU_{i,v}\subset K_S^n.\]
Since $\cM_S$ is discrete in $K_S$, there is an $\ell \geq 0$ such that $\bU_{i,S}\cap \cM_S^n = \{0\}$, for all $i\geq \ell$. We set $\cU$ to be the set of (isomorphic) images of the $\bU_{i,S}$ in $(K_S/\cM_S)^n$, for all $i\geq \ell$.
\end{enumerate}
\end{definition}

\begin{proposition}\label{P:TauEndomorphisms}
Let $(\cM, \cM^\infty)$ be a taming pair for $K/F$ and let $S$ be a finite set of primes of $F$ containing $S_\infty$.
\begin{enumerate}
\item Let $\phi = D\tau^\ell$ for some $D\in M_n(\cO_{F,S})$ and $\ell\geq 1$.  Then $\phi$ is a locally contracting endomorphism of $(K_S/\cM_S)^n$.
\item Any $\Phi \in M_n(\cO_{F,S})\{\tau\} [[Z]]\tau$ is a nuclear endomorphism of $(K_S/\cM_S)^n[[Z]]$.
\item Let $D \in M_n(\cO_{F,S})$, let $\alpha \in \cO_{F,S}$ and let $\phi = D\tau^\ell$ for $\ell\geq 1$. Then we have
$$\det\nolimits_{R[[Z]]}(1 + \alpha\phi Z^m\rvert K_S/\mathcal{M}_S) = \det\nolimits_{R[[Z]]}(1 + \phi \alpha Z^m\rvert K_S/\mathcal{M}_S),$$
for all $m\in\Z_{\geq 1}$.
\end{enumerate}
\end{proposition}

\begin{proof}
For $D\in M_n(\cO_{F,S})$, we denote the $(i,j)$ entry of $D$ by $D_{i,j}$.

The proof of part (1) follows the proof of \cite[Lemma 2.3.3]{FGHP20} with a minor alteration:  here we let $m\in\Z_{\geq 1}$, such that $$m\geq \max\{\frac{1 - \min(v(D_{i,j}))}{q-1}\mid v\in S\}.$$  Then $\tau(\mathcal M_v)\subseteq \mathcal M_v$ and $\tau(\mathcal M^\infty_v)\subseteq \mathcal M^\infty_v$, so we have $\phi(\bU_{i, S})\subseteq \bU_{i + 1, S}$, for all $i\geq m$.

The proof of part (2) is identical to that of \cite[Cor. 2.3.4]{FGHP20}.

The proof of part (3) follows as in the proof of \cite[Proposition 2.3.5]{FGHP20}, with one minor change: assuming that $D\ne 0$ (as the statement is clearly true for $D=0$),  in the proof in loc.cit. we set $a\in\Bbb Z_{\geq 0}$, such that $\bU_{a, S}\cap \cM_S^n=\{0\}$ and enlarge it if necessary, so that 
\[b:=\min\{a + \frac{\min (v(D_{i,j}))}{e_v}\mid v\in S\} > \max(\{0\}\cup\{-v(\alpha)\mid v\in S\}),\]
where $e_v:=v(t^{-1})$ if $v\in S_\infty$ and $e_v=1$ otherwise.
The rest of the arguments then carry through without modification.
\end{proof}

\begin{lemma}\label{L:independence-on-W} Let $(\cM, \cM^\infty)$ and $S$ as in the last proposition. Let $\Phi\in{\text M}_n(\cO_{F,S})\{\tau\}[[Z]]\tau$, viewed as
an $R[[Z]]$--endomorphism of $(K_S/\cM_S)^n[[Z]]$. Then the nuclear determinant
$${\rm det}_{R[[Z]]}(1+\Phi|(K_S/\cM_S)^n)$$
is independent of the taming pair $(\mathcal M, \mathcal M^\infty)$.
\end{lemma}
\begin{proof}
This is the $n$--dimensional analogue of Lemma 2.3.10 of \cite{FGHP20} and the proof is identical.
\end{proof}

\begin{lemma}\label{L:LocalizationLemma} Let $\mathcal{M}$ be a taming module for $K/F$, let $S$ be a finite set of primes of $F$ containing $S_\infty$, let $v\in {\rm MSpec}(\cO_F)\setminus S$, and let $S' := S\cup\{v\}$. Then, for any operator $\Phi\in\Mat_n(\cO_{F,S})\{\tau\}[[Z]]\tau Z$, we have
$$\det\nolimits_{R[[Z]]}(1 + \Phi\rvert (\mathcal{M}/v\mathcal{M})^n) = \frac{\det\nolimits_{R[[Z]]}(1 + \Phi\rvert (K_{S'}/\mathcal{M}_{S'})^n)}{\det\nolimits_{R[[Z]]}(1 + \Phi\rvert (K_S/\mathcal{M}_S)^n)}.
$$ \end{lemma}

\begin{proof}
The proof of this lemma follows the proof of \cite[Lemma 3.0.1]{FGHP20} very closely, with minor modifications.  Begin with the following exact sequence of compact, $G$-c.t. $\F_q[G]$-modules
\[0\longrightarrow (\cM_v)^n\overset{\psi^n}\longrightarrow\left (\frac{K_{S'}}{\cM_{S'}}\right )^n\overset{\eta^n}\longrightarrow\left (\frac{K_S}{\cM_S}\right )^n\longrightarrow 0,\]
where $\psi:\cM_v\to\frac{K_{S'}}{\cM_{S'}}$ and $\eta: \frac{K_{S'}}{\cM_{S'}}\to \frac{K_{S}}{\cM_{S}}$ are defined in loc.cit. Just as in loc.cit., one constructs bases of open neighborhoods $\mathfrak U'$ and $\mathfrak U$ for
$\left(\frac{K_{S'}}{\cM_{S'}}\right)^n$ and $\left(\frac{K_{S}}{\cM_{S}}\right)^n$, respectively, out of a taming pair $(\cM, \mathcal W^\infty)$. As in loc.cit, it is easy to check that $\mathfrak U'$ and $\mathfrak U$ satisfy the properties
in Proposition \ref{P:DeterminantFacts}(2) and conclude that
\[\det\nolimits_{R[[Z]]}\left (1 + \Phi\bigg\rvert \left (\frac{K_{S'}}{\mathcal{M}_{S'}}\right )^n\right ) = \det\nolimits_{R[[Z]]}(1 + \Phi\rvert (\mathcal{M}_v)^n)\cdot\det\nolimits_{R[[Z]]}\left (1 + \Phi\bigg\rvert \left (\frac{K_S}{\mathcal{M}_S}\right )^n\right ).\]
Finally, as in loc.cit., we note that
\[\det\nolimits_{R[[Z]]}(1 + \Phi\rvert (\mathcal{M}_v)^n) = \det\nolimits_{R[[Z]]}(1 + \Phi\rvert (\mathcal{M}/v\mathcal{M})^n),\]
which concludes the proof.
\end{proof}

\begin{theorem}[Trace Formula]\label{T:TraceFormula}
Let $\mathcal{M}$ be a taming module for $K/F$, let $S$ be a finite set of primes of $F$ containing $S_\infty$ and let $\Phi \in M_n(\cO_{F,S})\{\tau\}[[Z]] \tau Z$.
\[\prod_{v\in{\rm MSpec}(\cO_{F,S})}\det\nolimits_{R[[Z]]}\left (1 + \Phi\rvert \left (\mathcal{M}/v\mathcal{M}\right )^n\right ) = \det\nolimits_{R[[Z]]}\left (1 + \Phi\rvert \left (K_S/\mathcal{M}_S\right )^n\right )^{-1}.\]
\end{theorem}
\begin{proof}
This proof is similar to that of \cite[Thm. 3.0.2]{FGHP20}, although in the higher dimensional case we need to be careful, as the elements of $M_n(\cO_{F,S})$ don't commute with each other, in general.
As in \cite{FGHP20}, we prove the equality above $\mod Z^N$ (in $R[[Z]]/Z^N$), for all $N\in\Bbb Z_{\geq 0}$.

Write $\Phi = \sum_{m = 1}^\infty\phi_m Z^m$ with $\phi_m\in M_n(\OX_{F,S})\{\tau\}\tau$.  Let $N\in\Bbb Z_{>0}$ and let $D = D(N)$ be such that we have ${\rm deg}_\tau\phi_m<\frac{mD}{N}$, for all $m<N$. Let
\[T := T_{D(N)} := S\cup\{v\in\text{MSpec}(\cO_{F,S}) \mid  [\OX_{F,S}/v:\F_q]< D\}.\]
By Lemma \ref{L:LocalizationLemma} applied $|T\setminus S|$ times, we have
\[\prod_{v\in T\setminus S}\det\nolimits_{R[[Z]]}(1 + \Phi\rvert \left (\mathcal{M}/v\mathcal{M}\right )^n) =\frac{\det\nolimits_{R[[Z]]}(1 + \Phi\rvert (K_{T}/\mathcal{M}_{T})^n)}{\det\nolimits_{R[[Z]]}(1 + \Phi\rvert (K_{S}/\mathcal{M}_{S})^n)} .\]
Therefore, it suffices to show that
\[\prod_{v\in\text{MSpec}(\cO_{F,T})}\det\nolimits_{R[[Z]]/Z^N}(1 + \Phi\rvert \left (\mathcal{M}/v\mathcal{M}\right )^n)=\det\nolimits_{R[[Z]]/Z^N}(1 + \Phi\rvert \left (K_T/\mathcal{M}_T\right)^n )^{-1}.\]
In fact, we will show that both sides in the equality above are equal to $1$ in $R[[Z]]/Z^N$.

Following Taelman \cite[Thm. 3]{T12}, we let $\mathcal{S}_{D,N}\subseteq M_n(\OX_{F,T})\{\tau\}[[Z]]/Z^N$ be the set
$$\mathcal{S}_{D,N} = \left\{1 + \sum_{m = 1}^{N-1}\psi_mZ^m \bigg\vert  \deg_{\tau}(\psi_m) < \frac{mD}{N}, \text{ for all } m<N\right\}.$$
The set $\mathcal{S}_{D,N}$ is a group under multiplication, and $(1 + \Phi)\bmod Z^N\in\mathcal{S}_{D,N}$.  Observe that we can inductively decompose the operator $(1+\Phi)$ into a product of elementary terms mod $Z^N$.  Since the coefficients of $\Phi$ no longer commute as they did in \cite{T12} we briefly illustrate how this argument works.  Suppose we write
\[1+\Phi =1+ (P_{1,0} + P_{1,1}\tau + \dots + P_{1,\ell} \tau^\ell) Z + (P_{2,0} + P_{2,1}\tau + \dots + P_{2,m}\tau^m) Z^2 + \dots,\quad P_{i,j}\in M_n(\cO_{F,T}).\]
Then we find that mod $Z^2$, this expression for $1+\Phi$ agrees with
\begin{equation}\label{E:Zterms}
(1+P_{1,0} Z)(1+P_{1,1}\tau Z)(1+P_{1,2}\tau^2 Z) \dots (1+P_{1,\ell} \tau^\ell Z).
\end{equation}
We then include terms in the above product expansion which cancel out the $Z^2$ terms from \eqref{E:Zterms}, namely
\begin{equation}\label{E:Z^2cancelled}
(1-P_{1,0}P_{1,1}\tau Z^2)(1-P_{1,0}P_{1,2}\tau^2 Z^2) + \dots (1-P_{1,\ell-1}P_{1,\ell}\tau^{2\ell-1} Z^2)
\end{equation}
and include terms which give the correct coefficients for the $Z^2$ terms
\begin{equation}\label{E:Z^2terms}
(1+P_{2,0} Z^2)(1+P_{2,1}\tau Z^2)\dots (1+P_{2,m}).
\end{equation}
Putting this altogether we see that $1+\Phi$ agrees with the product of the three expressions \eqref{E:Zterms}-\eqref{E:Z^2terms} mod $Z^3$.  Then, we continue this process inductively so that we get an equality of the form
\[1+\Phi = \prod_{i,j,k} (1+M_{i}\tau^j Z^k) \bmod Z^N,\quad M_i \in M_n(\cO_{F,T}).\]

Now we use a trick of Anderson (\cite[Prop 9]{Anderson00}). Since $\OX_{F,T}$ has no residue fields of degree $d<D$ over $\Fq$, for every $d<D$ there exists $f_{dj}, a_{dj}\in\OX_{F,T}$, with $1\leq j\leq M_d$, such that
\[1 = \sum_{j = 1}^{M_d} f_{dj}(a_{dj}^{q^d} - a_{dj}).\]

Then for every $R\in M_n(\OX_{F,T})$, and every $n<N$ and $d<D$, we have
\[1 - R\tau^dZ^n\equiv \prod_{j = 1}^{M_d}\frac{1 - (Rf_{dj}\tau^d)a_{dj}Z^n}{1 - a_{dj}(Rf_{dj}\tau^d)Z^n}\bmod Z^{n+1},\]
since the constants $f_{dj}, a_{dj}$ commute with the matrix $R$.  Thus every such operator $1+\Phi$ may be decomposed mod $Z^N$ into a finite product with terms of the form
\[\frac{1 - (R\tau^d)aZ^n}{1 - a(R\tau^d)Z^n},\quad R\in M_n(\OX_{F,T}), a \in \OX_{F,T}.\]
Then, by properties of finite determinants together with Proposition \ref{P:TauEndomorphisms}(3) we get
\[\det\nolimits_{R[[Z]]/Z^N}\left(\frac{1 - (R\tau^d)aZ^n}{1 - a(R\tau^d)Z^n}\bigg\vert (\mathcal{M}/v\mathcal{M})^n\right) = \det\nolimits_{R[[Z]]/Z^N}\left(\frac{1 - (R\tau^d)aZ^n}{1 - a(R\tau^d)Z^n}\bigg\rvert (K_T/\mathcal{M}_T)^n\right) = 1,\]
for all $v\in\text{MSpec}(\cO_{F,T})$ which finishes the proof.
\end{proof}

\begin{corollary} \label{C:CorToTraceFormula} Let $\mathcal{M}$ be a taming module for $K/F$. Let $E$ be a $t$-module with structural morphism $\phi_E:\Fq[t]\to M_n(\cO_F\{\tau\})$. Then
$$\Phi = \frac{1 - \phi_E(t)T^{-1}}{1 - d_E[t]T^{-1}} - 1:= (1 - \phi_E(t)T^{-1})\cdot(1 - d_E[t]T^{-1})^{-1}-1\in M_n(\cO_{F})\{\tau\}[[T^{-1}]] $$
is a nuclear operator on $(K_\infty/\mathcal{M})^n[[T^{-1}]]$ and we have
\[\Theta_{K/F}^{E,\mathcal{M}}(0) = \det\nolimits_{R[[T^{-1}]]}(1 + \Phi\mid (K_\infty/\mathcal{M})^n)\rvert_{T = t}.\]
\end{corollary}

\begin{proof}
By Remark \ref{R:finite-det} applied to $V:=\left(\cM/v\right)^n$, we have an equality
\[ \Theta_{K/F}^{E,\mathcal{M}}(0) = \prod_{v}\frac{|\Lie_E(\cM/v)|_G}{|E(\cM/v)|_G}=\prod_{v} \frac{\det\nolimits_{R[[T^{-1}]]}(1 - d_E[t]T^{-1}\mid (\cM/v)^n)\rvert_{T = t}}{\det\nolimits_{R[[T^{-1}]]}(1 - \phi_E(t)T^{-1}\mid (\cM/v)^n)\rvert_{T = t}}.\]
Since

\[\Phi = \sum_{n = 1}^\infty (d_E[t] - \phi_E(t))d_E[t]^{n-1}T^{-n}\in M_n(\cO_{F})\{\tau\}[[T^{-1}]]\tau T^{-1},\]
by Proposition \ref{P:TauEndomorphisms}(2), $\Phi$ is a nuclear operator on $(K_\infty/\cM)^n$ and on $(\cM/v)^n$, for all $v$.  Applying the Trace Formula \ref{T:TraceFormula} for $S = S_\infty$ gives
\[\Theta_{K/F}^{E,\mathcal{M}}(0) =\prod_{v}\det\nolimits_{R[[T^{-1}]]}(1 + \Phi\mid(\cM/v)^n)^{-1}\rvert_{T = t} = \det\nolimits_{R[[T^{-1}]]}(1 + \Phi\mid (K_\infty/\cM)^n)\rvert_{T = t}.\]
\end{proof}

\section{Volume Formula}\label{S:Volume}

In this section, we prove a theorem linking certain nuclear determinants to volumes of objects in the Arakelov class $\mathcal C$. When combined with Corollary \ref{C:CorToTraceFormula}, this theorem provides the essential bridge between special values of $L$--functions and volumes of objects in the Arakelov class $\mathcal C$. This section is the generalization to $t$--modules of \S5 in \cite{FGHP20}.
\medskip

Let $M_1,M_2\in \cC$. Assume that their structural exact sequences are as follows
\begin{equation}\label{E:StructuralSequenceM_i}
0\to \Lie_E(K_\infty)/\Lambda_s\overset{\iota_s}\longrightarrow M_s\overset{\pi_s}\longrightarrow H_s\to 0,
\end{equation}
where $\Lambda_s$ is an $A[G]$--lattice in $\Lie_E(K_\infty)$, for $s=1,2$.
Fix $\ell>0$ sufficiently large so that $t^{-i}\cO_{K_\infty}^n\cap\Lambda_s=\{0\}$, for all $i\geq\ell$ and $s=1, 2$ and identify $t^{-i}\cO_{K_\infty}^n$ with its image in $\Lie_E(K_\infty)/\Lambda_s$, for all $i\geq \ell$. Fix an $\infty$--taming module $\mathcal W^\infty$ for $K/F$. With notations as above, we use Definition \ref{D:UNeighborhoods}(2) to construct open neighborhoods of $0$
in $M_s$, for all $s=1,2$. We denote these by $\cU = \{\bU_{i,\infty}\}_{i\geq \ell}$. By the construction above, there exists $a\in\Z_{>0}$, which we fix once and for all, such that
\begin{equation}\label{E:U-inclusions}
 t^{-a-i}\cO_{K_\infty}^n\subseteq \bU_{i, \infty}\subseteq t^{-i}\cO_{K_\infty}^n, \qquad\text{ for all }i\geq\ell.
\end{equation}

We need to understand how the $d_E[t]$-action on $\Lie_E(K_\infty)$ interacts with the chosen open neighborhoods $\bU_{i,\infty}$. We remind the reader that $d_E[t]=tI_n+N$, for some nilpotent matrix $N$. The extra degree contributions from $N$ cause most of the deviations from the proofs presented in \S5 of loc.cit.

\begin{definition}
We equip $K_\infty$ with the sup norm, normalized such that $\lVert t\rVert=q$,  then extend this to $\Mat_{r\times k}(K_\infty)$ by taking the max of the norms of the entries of a matrix. Abusively, we denote this extended norm by $\lVert \cdot \rVert$. \end{definition}
Note that in this norm, for $D, D' \in \Mat_{r\times k}(K_\infty)$ and $\bv \in K_\infty^k$ we have
\[\lVert D\bv\rVert \leq \lVert D \rVert \cdot \lVert \bv \rVert, \qquad \lVert D+D'\rVert\leq {\rm max}\{\lVert D\rVert, \lVert D'\rVert\},  \]
and also $\lVert D\cdot D'\rVert\leq \lVert D\rVert\cdot\lVert D'\rVert$ if $r=k$.

\begin{lemma}\label{L:Containments}
There exists a constant $C \in \Z_{\geq 0}$ such that
\begin{enumerate}
\item For all $m\in \Z$, we have
\[q^m = \lVert t^m\rVert \leq \lVert d_E[t]^m\rVert \leq \lVert t^m\rVert\cdot q^C = q^{m+C}.\]
\item For all $m\in \Z$, we have
\[t^{m-C}\cO_{K_\infty}^n\subset d_E[t]^m\cO_{K_\infty}^n\subset t^{m+C}\cO_{K_\infty}^n.\]
\item For all $i\geq \ell$, we have
\[t^{-a-i-2C} \cO_{K_\infty}^n \subset d_E[t]^{-a-i-C} \cO_{K_\infty}^n \subset t^{-a-i} \cO_{K_\infty}^n \subset \bU_{i,\infty} \subset t^{-i} \cO_{K_\infty}^n \subset d_E[t]^{-i+C}\cO_{K_\infty}^n.\]
\end{enumerate}

\end{lemma}
\begin{proof}
Recall that $d_E[t] = t \cdot \Id_n + N$, where $N$ is a nilpotent matrix such that $N^e = 0$ for some $e\leq n$.  Then for $m\in \Z$ we have
\begin{align*}
d_E[t]^m = (t \cdot\Id_n + N)^m &= t^m\cdot {\rm Id}_n + a_1 t^{m-1} N + \dots + a_{e-1} t^{m-e+1}N^{e-1}\\
&= t^{m}({\rm Id}_n + a_1 t^{-1} N + \dots + a_{e-1}t^{1-e} N^{e-1}),
\end{align*}
for some $a_i \in \F_q$.  Thus let $C\in \Z_{\geq 0}$, such that
\[q^C\geq \max_{0\leq k\leq e-1}{\lVert t^{-k} N^{k}}\rVert.\]
On the other hand, let $r\in\Bbb Z_{\geq 0}$, such that $N_{t^m}^{q^r}=0$. We have
$$\lVert d_E[t^m]\rVert^{q^r}\geq\lVert d_E[t^m]^{q^r}\rVert=\lVert t^{mq^r}I_n\rVert=\lVert t\rVert^{mq^r}.$$
Consequently, we have $\lVert d_E[t^m]\rVert\geq \lVert t^m\rVert$. So, since $d_E[t^m]=d_E[t]^m$, we have
\[q^m=\lVert t^m\rVert \leq \lVert d_E[t]^m\rVert \leq \lVert t^m\rVert\cdot q^C = q^{m+C},\]
for all $m\in\Bbb Z$. The right-side inequality above implies that
$$d_E[t]^m\cO_{K_\infty}^n\subset t^{m+C}\cO_{K_\infty}^n, \quad t^{-m-C}\cO_{K_\infty}^n\subset d_E[t]^{-m}\cO_{K_\infty}^n,$$
for all $m\in\Z$. Therefore, for all $i\geq\ell$, we have
\[t^{-a-i-2C} \cO_{K_\infty}^n \subset d_E[t]^{-a-i-C} \cO_{K_\infty}^n \subset t^{-a-i} \cO_{K_\infty}^n \subset \bU_{i,\infty} \subset t^{-i} \cO_{K_\infty}^n \subset d_E[t]^{-i+C}\cO_{K_\infty}^n.\]
\end{proof}

\begin{definition}\label{D:N-tangent} Let $M_1,M_2$ be in the class $\cC$ and have structural exact sequences as in  \eqref{E:StructuralSequenceM_i} and let $N\in\Z_{\geq 0}$. A continuous $R$--module morphism $\gamma:M_1\to M_2$ is called $N$-tangent to the identity if there exists $i\geq \ell$ such that
\begin{enumerate}
\item $\gamma$ induces a bijective isometry $(\iota_2^{-1}\circ\gamma\circ\iota_1): t^{-i}\cO_{K_\infty}^n\simeq t^{-i}\cO_{K_\infty}^n$, where $\iota_s$ is the structural map of \eqref{E:StructuralSequenceM_i}, for $s=1,2$.
\item If we let $\gamma_i$ denote the bijective isometry $(\iota_2^{-1}\circ\gamma\circ\iota_1): t^{-i}\cO_{K_\infty}^n\simeq t^{-i}\cO_{K_\infty}^n$, then
$$||\gamma_i(x)-x||\leq ||t||^{-N-a-2C}\cdot||x||, \qquad\text{ for all }x\in t^{-i}\cO_{K_\infty},$$
where $a$ and $C$ are as defined above.
\end{enumerate}
If $\gamma$ is $N$--tangent to the identity for all $N\geq 0$, $\gamma$ is called infinitely tangent to the identity.
\end{definition}

\begin{proposition}\label{P:Ntangpowerseries}
Let $\Gamma:K_\infty^n\to K_\infty^n$ be an $\F_q[G]$-linear map given by an everywhere convergent power series
\[\Gamma(z) = \text{\rm Id}_n\cdot \bz + D_1 \bz\twist +D_2 \bz\twistk{2} + \dots,\qquad\text{with } D_i \in {\text M}_n(K_\infty),\quad \bz \in K_\infty^n.\]
Let $\Lambda_1,\Lambda_2\subset K_\infty^n$ be $A[G]$-lattices and assume that $\Gamma(\Lambda_1) \subseteq \Lambda_2$. Let $\widetilde{\Gamma}:K_\infty^n/\Lambda_1\to K_\infty^n/\Lambda_2$ the map induced by $\Gamma$. Assume that $\gamma:M_1\to M_2$ is a continuous $\F_q[G]$--linear morphism such that $\iota_2^{-1}\circ\gamma\circ\iota_1=\widetilde\Gamma$ on $t^{-\ell}\cO_{K_\infty}^n$.  Then $\gamma$ is infinitely tangent to the identity.
\end{proposition}
\begin{proof}
Let $N\geq 1$. We will show that $\gamma$ is $N$--tangent to the identity.
Since the power series for $\Gamma$ is everywhere convergent, the coefficients $D_i$ must be bounded in norm.  Let $D:=\sup_i||D_i||.$
Thus, if $i\geq \ell$ is sufficiently large and $\bz\in t^{-i}\cO_{K_\infty}^n$, then we have
$$\lVert(\iota_2^{-1}\circ\gamma\circ\iota_1)(\bz) - \bz\rVert = \lVert(D_1\bz^q + D_2 \bz^{q^2} + \cdots)\rVert\leq D\cdot||\bz||^q, \quad \lVert(\iota_2^{-1}\circ\gamma\circ\iota_1)(\bz)\rVert=\lVert \bz\rVert.$$
In particular, if $i$ is sufficiently large, then $(\iota_2^{-1}\circ\gamma\circ\iota_1): t^{-i}\cO_{K_\infty}^n\to t^{-i}\cO_{K_\infty}^n$ is an isometry, which is strictly
differentiable at $0$ and $(\iota_2^{-1}\circ\gamma\circ\iota_1)'(0)=1$. By the non-archimedean inverse function theorem (see \cite[2.2]{Igusa}), for all $i\gg\ell$ the map $(\iota_2^{-1}\circ\gamma\circ\iota_1): t^{-i}\cO_{K_\infty}^n\simeq t^{-i}\cO_{K_\infty}^n$ is a bijective isometry.
Further, for all $i\gg \ell$ and all $\bz\in t^{-i}\cO_{K_\infty}^n\setminus \{0\}$, we have
$$\frac{\lVert(\iota_2^{-1}\circ\gamma\circ\iota_1)(\bz) - \bz\rVert}{\lVert \bz\rVert}\leq D\lVert \bz\rVert^{q-1}\leq D\lVert t\rVert^{-i(q-1)}\leq\lVert t\rVert^{-N-a-2C},$$
which shows that, indeed, $\gamma$ is $N$--tangent to the identity.
\end{proof}

\begin{definition} Let $M_1, M_2\in\mathcal C$ and let $\gamma:M_1\simeq  M_2$ be an $R$--linear topological isomorphism. We define the endomorphism $\Delta_\gamma$ of $\power{M_1}{T\inv}$ by
\begin{equation}\label{D:Deltagamma}
\Delta_\gamma := \frac{1 - \gamma\inv t \gamma T\inv}{1-tT\inv} - 1=\sum_{m=1}^\infty \delta_m T^{-m},
\end{equation}
where $\delta_m = (t-\gamma\inv t \gamma )t^{m-1}$, for all $m\geq 1.$
\end{definition}

\begin{lemma}With notations as in the definition above,
if the topological $R$--linear isomorphism $\gamma:M_1\simeq M_2$ is $N$-tangent to the identity, then the map $(\Delta_\gamma$ mod $T^{-N})$ is a nuclear endomorphism of $\power{M_1}{T\inv}/T^{-N}$. Furthermore, if $\gamma$ is infinitely tangent to the identity, then $\Delta_\gamma$ is a nuclear endomorphism of $M_1[[T^{-1}]]$.
\end{lemma}
\begin{proof}
We use the arguments in \cite[Lemma 5.1.5]{FGHP20}, except for a few norm estimates which must be slightly adjusted. Let $1\leq m<N$. Let $i$ and $\gamma_i$ be as in Definition \ref{D:N-tangent}. Then it is easy to check that for all $j\geq m+2C+i$, we have equalities of functions defined on  $t^{-j}\cO_{K_\infty}^n$
\[\delta_{m}=(d_E[t]-\gamma_i^{-1}d_E[t]\gamma_i)d_E[t]^{m-1}=\gamma_i^{-1}(\gamma_i-1)d_E[t]^m+\gamma_i^{-1}d_E[t](1-\gamma_i)d_E[t]^{m-1}.\]
Consequently, the conditions imposed upon $\gamma_i$ in Definition \ref{D:N-tangent}(2) and the inclusions in Lemma \ref{L:Containments}(2) imply that
\[(\gamma_i^{-1}(\gamma_i-1)d_E[t]^m)(t^{-j}\cO_{K_\infty}^n)\subset t^{-a-j-1-C}\cO_{K_\infty}^n,\]
\[(\gamma_i^{-1}d_E[t](1-\gamma_i)d_E[t]^{m-1})(t^{-j}\cO_{K_\infty}^n)\subset t^{-a-j-1}\cO_{K_\infty}^n. \]
Thus the inclusions \eqref{E:U-inclusions} imply that $\delta_{m}(\bU_{j,\infty})\subset \bU_{j+1, \infty}$, which concludes the proof.
\end{proof}

Now, we treat the particular case $M_1=M_2=\Lie_E(K_\infty)/\Lambda$, for an $A[G]$--projective lattice $\Lambda\subseteq \Lie_E(K_\infty)$. As above, we fix $\ell>0$ such that $t^{-\ell}\cO_{K_\infty}^n\cap\Lambda=\{0\}$ and fix $a\in\Z_{>0}$ satisfying \eqref{E:U-inclusions} and $C\in\Z_{\geq 0}$ satisfying Lemma \ref{L:Containments}. For simplicity, we let $V:=\Lie_E(K_\infty)/\Lambda$.

\begin{definition}\label{D:contraction}
An $R$--linear, continuous endomorphism $\phi: V\to V$ is called a local $M$--contraction, for some
$M\in\Bbb Z_{>0}$, if there exists $i\geq\ell$ such that
$$||\phi(x)||\leq ||t||^{-M}\cdot||x||, \text{ for all }x\in t^{-i}\cO_{K_\infty}^n.$$
\end{definition}

\begin{remark}\label{R:contraction-nuclear}
If $\phi$ as above is a local $M$--contraction for some $M>a$, then $\phi$ is locally contracting on $V$ and therefore
the nuclear determinant ${\rm det}_{R[[Z]]}(1-\phi\cdot Z|V)$ makes sense.  Indeed, pick an $i>\ell$ as in the definition above.
Then, for all $j\geq i$, we have
$$\phi(\bU_{j,\infty})\subseteq\phi(t^{-j}\cO_{K_\infty}^n)\subseteq t^{-j-M}\cO_{K_\infty}^n\subseteq t^{-j-a-1}\cO_{K_\infty}^n\subseteq \bU_{j+1, \infty}.$$
This shows that $\bU_{i, \infty}$ is a nucleus for $\phi$.
\end{remark}

\begin{proposition}\label{P:CommutingOperators} Let $\gamma:V\simeq   V$ be an $R$--linear, continuous isomorphism which is $N$--tangent to the identity.
Let $\alpha := d_E[t]\gamma$ and let $\psi:V\to V$ be an $R$-linear, continuous, local $M$-contraction for $M>2a+5C$.  Let $D:=M-(1+C).$ Then the following hold.
\begin{enumerate}
\item $\alpha\psi$ and $\psi\alpha$ are local $D$--contractions on $V$.
\item ${\rm det}_{R[[Z]]}(1-\alpha\psi\cdot Z|V)={\rm det}_{R[[Z]]}(1-\psi\alpha\cdot Z|V).$
\end{enumerate}
\end{proposition}
\begin{proof}
Again, some of the degree estimates in the proof are slightly different than those in \cite[Prop. 5.2.3]{FGHP20}, so we give the details here.
Fix $i>\ell$ such that $\gamma:t^{-(i-1-C)}\cO_{K_\infty}^n\to t^{-(i-1-C)}\cO_{K_\infty}^n$ is a bijective isometry and such that
\[||\psi(x)||\leq ||t||^{-M}\cdot ||x||, \text{ for all }x\in t^{-(i-1-C)}\cO_{K_\infty}^n.\]

(1) For $i$ chosen as above it is easy to check that
$$||\alpha\psi(x)||\leq ||t||^{-D}||x||, \quad ||\psi\alpha(x)||\leq ||t||^{-D}||x||, \text{ for all }x\in t^{-i}\cO_{K_\infty}.$$
So, $\alpha\psi$ and $\psi\alpha$ are $D$--contractions on $t^{-i}\cO_{K_\infty}$. Since $D>a$, these are locally contracting endomorphisms of $V$ by Remark \ref{R:contraction-nuclear}, and so the nuclear determinants in (2) make sense.
\medskip

(2) Since $M, D>a$, Remark \ref{R:contraction-nuclear} combined with the proof of (1) above show that $\psi$, $\alpha\psi$, and  $\psi\alpha$
are all locally contracting on $V$ of common nuclei $\bU_{j, \infty}$, for all $j\geq i$ and $i$ chosen as above.
Now, since $\gamma$ is an isomorphism and $V$ is $t$--divisible (because $\Lie_E(K_\infty)$ is, as $d_E[t]\in{\rm GL}_n(K_\infty)$), $\alpha$ is surjective. Therefore $\alpha$ induces an $R$--module isomorphism
$$V/\alpha^{-1}(\bU_{i,\infty})\overset{\alpha}\simeq V/\bU_{i, \infty}.$$
Since $\Lambda\cap \bU_{i, \infty}=\{0\}$, if we let $\alpha^{-1}(\bU_{i, \infty})^\ast:=\gamma^{-1}(d_E[t]\inv \bU_{i,\infty})$, we have
$$\alpha^{-1}(\bU_{i, \infty})=\gamma^{-1}(d_E[t]\inv\Lambda/\Lambda)\oplus \alpha^{-1}(\bU_{i, \infty})^\ast.$$
Since $\gamma$ and $d_E[t]$ are isomorphisms, the $R$--modules $\gamma^{-1}(d_E[t]\inv \bU_{i,\infty})$, $\gamma^{-1}(d_E[t]\inv\Lambda/\Lambda)$ and $\alpha^{-1}(\bU_{i, \infty})$ are projective
(i.e. $G$--c.t.) because $\bU_{i,\infty}$ and $\Lambda$ are $G$--c.t. By Lemma \ref{L:Containments} we get
\begin{equation}\label{E:alpha-plus}
t^{-(i+1)-a-C}\cO_{K_\infty}^n\subseteq \alpha\inv(\bU_{i, \infty})^\ast\subseteq t^{-(1+i)+C}\cO_{K_\infty}^n,\quad t^{-(i+1)-a}\cO_{K_\infty}^n\subseteq \bU_{i+1, \infty}\subseteq t^{-(i+1)}\cO_{K_\infty}^n.
\end{equation}
Consequently, we get
\begin{equation}\label{E:psi-alpha}
(\psi\alpha)(\alpha^{-1}(\bU_{i, \infty}))=\psi(\bU_{i, \infty})\subseteq t^{-i-M}\cO_{K_\infty}^n\subseteq t^{-(i+1)-a-C}\cO_{K_\infty}^n\subseteq \alpha^{-1}(\bU_{i, \infty})^\ast\subseteq \alpha^{-1}(\bU_{i, \infty}).
\end{equation}
So, we have a commutative diagram of morphisms of finite, projective $R$--modules
$$\xymatrix{
V/\alpha^{-1}(\bU_{i, \infty})\ar@{>}[r]^{\quad \alpha}_{\quad \sim}\ar@{>}[d]^{\,\psi\alpha} & V/\bU_{i, \infty}\ar@{>}[d]^{\,\alpha\psi} \\
V/\alpha^{-1}(\bU_{i, \infty})\ar@{>}[r]^{\quad \alpha}_{\quad \sim} &V/\bU_{i, \infty} ,}$$
whose horizontal maps are isomorphisms. This gives an equality of (regular) determinants
\begin{equation}\label{E:det2}{\rm det}_{R[[Z]]}(1-\alpha\psi\cdot Z|V/\bU_{i, \infty})={\rm det}_{R[[Z]]}(1-\psi\alpha\cdot Z|V/\alpha^{-1}(\bU_{i, \infty})).\end{equation}
Now, consider the short exact sequence of projective $R$--modules
$$0\to {\alpha^{-1}(\bU_{i, \infty})}/{\alpha^{-1}(\bU_{i, \infty})^\ast}\to V/{\alpha^{-1}(\bU_{i, \infty})^\ast}\to V/{\alpha^{-1}(\bU_{i, \infty})}\to 0.$$
Noting that \eqref{E:psi-alpha} implies that $\psi\alpha$ induces an $R$--linear endomorphism of the exact sequence above and that $\psi\alpha\equiv 0$ on ${\alpha^{-1}(\bU_{i, \infty})}/{\alpha^{-1}(\bU_{i, \infty})^\ast}$,
the exact sequence above gives
\begin{equation}\label{E:det3}{\rm det}_{R[[Z]]}(1-\psi\alpha\cdot Z|V/\alpha^{-1}(\bU_{i, \infty}))={\rm det}_{R[[Z]]}(1-\psi\alpha\cdot Z|V/\alpha^{-1}(\bU_{i, \infty})^\ast).
\end{equation}

Now, enlarge $i$ further so that $\psi\alpha$ is a $D$--contraction on $t^{-(i+1)+C}\cO_{K_\infty}^n$ (see proof of part (1).) Since $D\geq 2a+4C$, equation \eqref{E:alpha-plus} leads to the following inclusions
$$\psi\alpha(\alpha^{-1}(\bU_{i, \infty})^\ast),\, \psi\alpha(\bU_{i+1,\infty})\subseteq t^{-2a-3C-(i+1)}\cO_{K_\infty}^n\subseteq t^{-a-2C}\left(\alpha^{-1}(\bU_{i, \infty})^\ast\right).$$
Now, since \eqref{E:alpha-plus} also implies that
$$t^{-a-2C}\left(\alpha^{-1}(\bU_{i, \infty})^\ast\right)\subseteq \bU_{i+1,\infty},\, \alpha^{-1}(\bU_{i, \infty})^\ast,$$
the last displayed inclusions show that $\psi\alpha\equiv 0$ on the quotients
$$\bU_{i+1,\infty}/t^{-a-2C}\left(\alpha^{-1}(\bU_{i, \infty})^\ast\right),\quad \alpha^{-1}(\bU_{i, \infty})^\ast/t^{-a-2C}\left(\alpha^{-1}(\bU_{i, \infty})^\ast\right).$$ Consequently, a short exact sequence argument similar to the one used to prove \eqref{E:det3} above gives the following equalities of (regular) determinants
\begin{eqnarray*}
 {\rm det}_{R[[Z]]}(1-\psi\alpha\cdot Z|V/\alpha^{-1}(\bU_{i, \infty})^\ast)  &=& {\rm det}_{R[[Z]]}(1-\psi\alpha\cdot Z|V/t^{-a}\alpha^{-1}(\bU_{i, \infty})^\ast) \\
  &=& {\rm det}_{R[[Z]]}(1-\psi\alpha\cdot Z|V/\bU_{i+1,\infty}).
\end{eqnarray*}
Now, we combine these equalities with \eqref{E:det2} and \eqref{E:det3} to obtain
$${\rm det}_{R[[Z]]}(1-\alpha\psi\cdot Z|V/\bU_{i, \infty})={\rm det}_{R[[Z]]}(1-\psi\alpha\cdot Z|V/\bU_{i+1,\infty}).$$
Recalling that $\bU_{i, \infty}$ and $\bU_{i+1, \infty}$ are common nuclei for $\psi\alpha$ and $\alpha\psi$, this leads to the desired equality
of nuclear determinants, which concludes the proof of part (2).
\end{proof}

\begin{lemma}\label{L:MonomialsNtangent}
Let $\psi$, $\gamma$, $\alpha$ and $M$ be as in Proposition \ref{P:CommutingOperators}, except that we require that $M>((C+1)N)$, for some fixed $N \geq 1$.  Let $\beta$ be in the $R$--subalgebra $R\{\alpha, \psi\}$ of ${\rm End}_R(V)$ generated by $\alpha$ and $\psi$ with the property that it is a sum of monomials containing at least one factor of $\psi$ and of degree at most $m$, for some $m\leq N$.  Then $\beta$ is a local $[M-(C+1)(m-1)]$-contraction.
\end{lemma}

\begin{proof} Take $i\gg 0$, such that $\gamma$ is a bijective isometry on $t^{-i}\cO_{K_\infty}^n$ and $\psi$ is an $M$--contraction on $t^{-i+(C+1)(m-1)}\cO_{K_\infty}^n$
Then, for all $x \in t^{-i}\cO_{K_\infty}^n$, we have
\[\lVert \alpha(x)\rVert = \lVert d_E[t]\gamma (x)\rVert \leq \lVert t\rVert ^{C+1} \lVert x \rVert.\]  Consequently, since $\beta$ contains at least one occurrence of $\psi$, we also have
\[\Vert \beta(x) \rVert  \leq \lVert t\rVert ^{(C+1)(m-1) - M}\lVert x\rVert.\]
\end{proof}

\begin{corollary}\label{C:det=1}
Let $N\geq a$ and let $\gamma:V\simeq V$ be an $R$--linear, continuous isomorphism which is $((C+2)(N+1))$--tangent to the identity. Then, we have
$${\rm det}_{R[[T^{-1}]]/T^{-N}}(1+\Delta_\gamma\,|\, V[[T^{-1}]]/T^{-N})=1.$$
\end{corollary}
\begin{proof} We use the main ideas in the proof of Corollary 1 in \cite{T12}. Let $Z:=T^{-1}$.
Let $\alpha:=d_E[t]\gamma$ and $\psi:=(\gamma^{-1}-1)$, viewed as a continuous, $R$--linear endomorphism of $V$. Then, we have
$$1+\Delta_\gamma=\frac{1-(\psi+1)\alpha\cdot Z}{1-\alpha(\psi+1)\cdot Z}.$$
Now, since $\gamma^{-1}$ is $((C+2)(N+1))$--tangent to the identity, $\psi$ is a local $((C+2)(N+1)+a+2C)$--contraction (see Definition \ref{D:N-tangent}(2)).
As in the proof of Cor. 1 \cite{T12}, one writes
$$\frac{1-(\psi+1)\alpha\cdot Z}{1-\alpha(\psi+1)\cdot Z}\mod Z^N=\prod_{m=1}^{N-1}\left(\frac{1-\psi_m\alpha\cdot Z^m}{1-\alpha\psi_m\cdot Z^m}\right)\mod Z^N,$$
where the $\psi_m$'s are uniquely determined polynomials in $R\{\alpha, \psi\}$ of degree at most $m$, containing at least one factor of $\psi$.
According to Lemma \ref{L:MonomialsNtangent}, since
\[((C+2)(N+1)+a+2C) - (C+1)(m-1) \geq ((C+2)(N+1)+a+2C) - (C+1)(N-2)> 2a+5C,\]
we conclude that $\psi_m$ is a local $M_m$--contraction on $V$, with $M_m>2a+5C$, for all $m<N$.
Now, we may apply Proposition \ref{P:CommutingOperators}(2) to $\alpha$, $\psi:=\psi_m$ and $M_m$ to conclude that
$${\rm det}_{R[[Z]]/Z^N}(1+\Delta_\gamma\,|\,V[[Z]]/Z^N)=\prod_{m=1}^{N-1}{\rm det}_{R[[Z]]/Z^N}\left(\frac{1-\psi_m\alpha\cdot Z^m}{1-\alpha\psi_m\cdot Z^m}\,\bigg|\,V[[Z]]/Z^N\right)=1.$$
\end{proof}

\begin{lemma}[Independence of $\gamma$]\label{L:Independence of Gamma}
Let $M_1$ and $M_2$ be modules from the class $\mathcal C$ and let $\gamma_1,\gamma_2:M_1\to M_2$ be be two $R$-linear, continuous isomorphisms which are $((C+2)(N+1))$--tangent to the identity for some $N\geq a$.  Then
\[\det\nolimits_{R[[T\inv]]/T^{-N}}(1+\Delta_{\gamma_1}|M_1) = \det\nolimits_{R[[T\inv]]/T^{-N}}(1+\Delta_{\gamma_2}|M_1).\]
\end{lemma}

\begin{proof}
The proof is identical to the proof of \cite[Lemma 5.3.4]{FGHP20}.  We sketch a few of the details here for the benefit of the reader.  First we write
\begin{equation}\label{E:Delta-compose}(1+\Delta_{\gamma_1})=[\gamma_2^{-1}(1+\Delta_{\gamma_1\gamma_2^{-1}})\gamma_2]\cdot(1+\Delta_{\gamma_2}).\end{equation}
Thus, by exactness of determinants it suffices to show that
\[\det\nolimits_{R[[T\inv]]/T^{-N}}(1+\Delta_{\gamma_1\gamma_2^{-1}}\,|\,M_2)=1,\]
which follows from Corollary \ref{C:det=1} applied to $V=M_2$ and $\gamma = \gamma_1\gamma_2\inv$.
\end{proof}
\medskip

Next, we state and prove the volume formula. This is a generalization to the entire Arakelov class $\mathcal C$  of the fact that if $\gamma:H_1\simeq H_2$ is an $R$--linear isomorphism of {\it finite}, projective $R[t]$--modules (i.e. {\it finite} objects in class $\mathcal C$), then
\[\det\nolimits_{\power{R}{T\inv}}\left (1 + \Delta_\gamma \big|H_1\right )\bigg\vert_{T=t} = \frac{|H_2|_G}{|H_1|_G}.\]
(See \S5.3 of \cite{FGHP20} for a proof of the above formula.)

\begin{theorem}[Volume formula]\label{T:VolumeFormula}
Let $M_1$ and $M_2$ be modules form the class $\mathcal C$ and let $\gamma:M_1\to M_2$ be an $R$-linear, continuous isomorphism which is infinitely tangent to the identity, and assume that $M_2 = \Lie_E(K_\infty)/\Lambda_2$, for a projective $R[t]=A[G]$--lattice $\Lambda_2$ in $\Lie_E(K_\infty)$. Then
\[\det\nolimits_{\power{R}{T\inv}}\left (1 + \Delta_\gamma \big|M_1\right )\bigg\vert_{T=t} = \frac{\Vol(M_2)}{\Vol(M_1)}.\]
\end{theorem}

\begin{proof}
This proof follows nearly identically to that of \cite[Thm. 5.3.2]{FGHP20}.  For the convenience of the reader, we give a sketch of the proof here and refer the reader to loc.cit. for more details.  We remind the reader that $M_i$ fits in the structural exact sequence \eqref{E:StructuralSequenceM_i}.

We first establish the equality above $\mod t^{-N}$ in the restricted case that the $A[G]$-lattices $\Lambda_1$ and $\Lambda_2$ are both contained in a $(M_1,s_1)$--admissible $A[G]$-lattice $\Lambda\subset \Lie_E(K_\infty)$, but with the milder requirement that $\gamma$ is merely $((C+2)(N+1))$--tangent to the identity, for some $N\geq a$.  In this case, as in loc.cit., we pick an $R$-projective, open submodule $\cU$ of $\Lie_E(K_\infty)$, such that we have an equality of $R$-modules
\[\Lie_E(K_\infty) = \Lambda \oplus \cU.\]
Note that for such a $\cU$ the map $\gamma$ is an $R$-linear isomorphism
\[\gamma: M_1=\cU\oplus(\Lambda/\Lambda_1\times s_1(H_1))\simeq  \cU\oplus \Lambda/\Lambda_2=M_2.\]
This implies, as in loc.cit., that there is an $R$--module isomorphism
$$\xi: (\Lambda/\Lambda_1\times s_1(H_1))\simeq \Lambda/\Lambda_2.$$
Then, we fix an $R$--module isomorphism $\rho: M_1\to M_2$ such that $\rho$ is the identity on $\cU$ and it restricts to $\xi$ on $(\Lambda/\Lambda_1\times s_1(H_1))$.  Since $\gamma$ is $((C+2)(N+1))$--tangent to the identity and $\rho$ is infinitely tangent to the identity, Lemma \ref{L:Independence of Gamma} gives
\[{\rm det}_{R[[T^{-1}]]/T^{-N}}(1+\Delta_\gamma\,|\, M_1)={\rm det}_{R[[T^{-1}]]/T^{-N}}(1+\Delta_\rho\,|\, M_1).\]
Now, we get a commutative diagram of topological morphisms of modules in the class $\mathcal C$
$$\xymatrix{
0\ar[r] & (\Lambda/\Lambda_1\times s_1(H_1))\ar[r]\ar[d]^{}_{\wr}^{\xi} & M_1\ar[r]\ar[d]^{\rho}_{\wr} & K_\infty/\Lambda\ar[r]\ar[d]^{{\rm id}}_{=} & 0 \\
0\ar[r] & \Lambda/\Lambda_2\ar[r] & M_2\ar[r] & K_\infty/\Lambda\ar[r] & 0,}$$
from which we conclude that
\begin{eqnarray}
  \nonumber {\rm det}_{R[[T^{-1}]]/T^{-N}}(1+\Delta_\rho\,|\, M_1)\bigg\vert_{T=t}&=&{\rm det}_{R[[T^{-1}]]/T^{-N}}(1+\Delta_\xi\,|\,\Lambda/\Lambda_1\times s_1(H_1))\bigg\vert_{T=t} \\
   \nonumber &=&\frac{|(\Lambda/\Lambda_1\times s_1(H_1))|_G}{|\Lambda/\Lambda_2|_G}=\frac{{\rm Vol}(M_2)}{{\rm Vol}(M_1)}\mod t^{-N}.
\end{eqnarray}
This concludes the proof of our equality $\mod t^{-N}$ in the restrictive case.
\medskip

Now, we prove the general case.  Fix $N>a$, as above. For $i=1,2$, fix $A[G]$--free admissible lattices $\tilde\Lambda_i\supset \Lambda_i$ and set $X\in \GL_m(k_\infty[G])$ to be the change of basis matrix between two fixed $A[G]$--bases ${\bf e}_1$ and ${\bf e}_2$ for $\tilde \Lambda_1$ and $\tilde \Lambda_2$, respectively.  Then, we factor $X$ (see loc.cit.) as
\[X=B\cdot X_0,\qquad   X_0\in (1+t^{-((C+2)(N+1))}\Mat_m(\Fq[[t^{-1}]][G]),\quad B\in{\rm GL}_m(\Fq(t)[G]),\]
and let $\phi_{X_0}:\Lie_E(K_\infty) \to \Lie_E(K_\infty)$ be the $k_\infty[G]$-linear isomorphism given by $X_0$ in the $k_\infty[G]$--bases ${\bf e_1}$ and ${\bf e_2}$ of $\Lie_E(K_\infty)$.  This map, together with the structural exact sequence for $M_1$ induces the commutative diagram of objects in the Arakelov class $\mathcal C$
\begin{equation}\label{E:push-out}
\xymatrix{
0\ar[r] & \Lie_E(K_\infty)/\Lambda_1\ar[r]\ar[d]^{\phi_{X_0}}_{\wr} & M_1\ar[r]^{}\ar[d]^{\phi}_{\wr} & H_1\ar[r]\ar[d]^{{\rm id}}_{=} & 0 \\
0\ar[r] & \Lie_E(K_\infty)/\phi_{X_0}(\Lambda_1)\ar[r] & M_1'\ar[r]^{} & H_1\ar[r] & 0,
}
\end{equation}
where the bottom exact sequence is the push-out along $\phi_{X_0}$ of the upper one and $\phi$ is the map induced by $\phi_{X_0}$.  It follows that $\phi:M_1\to M_1'$ is $((C+2)(N+1))$-tangent to the identity and thus $\gamma\circ \phi\inv:M_1'\to M_2$ is as well.  Since $\phi_{X_0}(\Lambda_1)$ and $\Lambda_2$ are contained in a common lattice (because $\phi_{X_0}(\Lambda_1)$ and $\tilde\Lambda_2$ are, as the transition matrix between their $A[G]$--bases $\phi_{X_0}({\bf e_1})$ and ${\bf e_2}$ is $B\in {\rm GL}_m(k[G])$), we apply our result in the restrictive case to conclude that
\[{\rm det}_{R[[T^{-1}]]/T^{-N}}(1+\Delta_{\gamma\circ\phi^{-1}}\,|\, M'_1)|_{T=t} = \frac{{\rm Vol}(M_2)}{{\rm Vol}(M'_1)}\mod t^{-N}.\]
Now, since $\phi$ is $R[t]$-linear (i.e. $A[G]$--linear), we have $(1+\Delta_{\phi^{1}})=1$. Therefore, \eqref{E:Delta-compose} combined with the above equality gives
\[{\rm det}_{R[[T^{-1}]]/T^{-N}}(1+\Delta_{\gamma}\,|\, M_1)|_{T=t} = \frac{{\rm Vol}(M_2)}{{\rm Vol}(M'_1)}\mod t^{-N}.\]
Now, just as in loc.cit, it is easy to see that
\[\frac{{\rm Vol}(M_1)}{{\rm Vol}(M'_1)}={\rm det}(X_0)\equiv 1\mod t^{-N}.\]
Combining the last two equalities and taking the limit as $N\to \infty$ concludes the proof.
\end{proof}

\section{Main Results}\label{main-results-section}
\subsection{The Equivariant Tamagawa Number Formula for $t$--modules}\label{ETNF-section} In this section, we prove the results linking the special values at $s=0$ of the $L$--functions
associated to an abelian $t$--module $E$ as above and the volumes of the relevant objects in the Arakelov class $\mathcal C$. These are perfect generalizations to abelian $t$--modules of the main results in \cite{FGHP20}. (See \S6.1 in loc.cit.)

\begin{theorem}[ETNF for $t$-modules]\label{T:ETNFtMods}
Let $\cM$ be a taming module for $\cO_K/\cO_F$ and let $E$ be an abelian $t$-module of structural morphism
$\phi_E:A\to{\text M}_n(\cO_F)\{\tau\}$. Then, we have the following equality in $(1+t^{-1}\F_q[[t^{-1}]][G])$,
\[\Theta_{K/F}^{E, \cM}(0)=\frac{{\rm Vol}(E(K_\infty)/E(\cM))}{{\rm Vol}(\Lie_E(K_\infty)/\Lie_E(\cM))}.\]
\end{theorem}

\begin{proof}
Note that we have exact sequences
\begin{equation}\label{E:ExactSequenceM_1}
0\to \Lie_E(K_\infty)/\Exp_E\inv(E(\cM)) \overset{\Exp_E}\longrightarrow E(K_\infty)/E(\cM)\to H(E/\cM)\to 0,
\end{equation}
and
\begin{equation}\label{E:ExactSequenceM_2}
0\to \Lie_E(K_\infty)/\Lie_E(\cM)\overset{\iota_2={\rm id}}\longrightarrow\Lie_E(K_\infty)/\Lie_E(\cM)\to 0\to 0,
\end{equation}
and thus $\Lie_E(K_\infty)/\Lie_E(\cM)$ (trivially) and $E(K_\infty)/E(\cM)$ are both in the Arakelov class $\cC$ (see Definition \ref{D:classC}.) Let $\gamma: E(K_\infty)/E(\cM) \to \Lie_E(K_\infty)/\Lie_E(\cM)$ be the continuous $\F_q[G]$-module isomorphism given by the identity. Note that this map is \textit{not} $t$-linear, since the domain and image modules have different $t$-actions. Then, as in Definition \ref{D:N-tangent}, define $\tilde\gamma = \iota_2\inv\circ \,\gamma \circ \Exp_E$ and note that $\tilde\gamma: K_\infty^n/\Exp_E\inv(\cM)\to K_\infty^n/\Lie_E(\cM)$ is equal to the map induced by $\Exp_E$. Then, since $\Exp_E:K_\infty^n \to K_\infty^n$ is given by an everywhere convergent power series, Proposition \ref{P:Ntangpowerseries} implies that $\gamma$ is infinitely tangent to the identity. Thus, by Theorem \ref{T:VolumeFormula} we get
\[\det\nolimits_{\F_q[G][[T\inv]]}(1+\Delta_\gamma |E(K_\infty)/E(\cM))\bigg|_{T=t} = \frac{\Vol(\Lie_E(K_\infty)/\Lie_E(\cM))}{\Vol(E(K_\infty)/E(\cM))}.\]
By the definition of $\Delta_\gamma$ from \eqref{D:Deltagamma}, we can rewrite the above equality as
\[{\rm det}_{\F_q[G][[T\inv]]}\left(\frac{1-d_E[t]\cdot T^{-1}}{1-\varphi_E(t)\cdot T^{-1}}\,\bigg|\,\frac{ K_\infty^n}{\cM^n}\,\right )\Bigg|_{T=t}=\frac{\Vol(\Lie_E(K_\infty)/\Lie_E(\cM))}{\Vol(E(K_\infty)/E(\cM))}.\]
Finally, Corollary \ref{C:CorToTraceFormula} gives
\[\Theta_{K/F}^{E, \cM}(0)=\frac{{\rm Vol}(E(K_\infty)/E(\cM))}{{\rm Vol}(\Lie_E(K_\infty)/\Lie_E(\cM))}.\]
\end{proof}

As in \cite[Cor. 6.1.2]{FGHP20}, if $p\nmid |G|$,  we get a corollary to the preceding theorem, also obtained by Fang with different methods in \cite{Fang22}. (See Theorem 1.12 in loc.cit.)

\begin{corollary}\label{main-theorem-not-p}
If $p\nmid |G|$, then we have the following equality in $(1+t^{-1}\F_q[[t^{-1}]][G])$:
$$\Theta_{K/F}^E(0)=[\Lie_E(\cO_K): {\rm exp}^{-1}_E(E(\cO_K))]_G\cdot |H(E/\cO_K)|_G.$$
\end{corollary}
\begin{proof}
This follows immediately from Theorem \ref{T:ETNFtMods}. Indeed, if $p\nmid |G|$, then $\cO_K$ is a taming module for $K/F$, all the modules of \eqref{E:ExactSequenceM_1} are $\Bbb F_q[G]$--projective, and $\Lie_E(K_\infty)/{\rm Exp}^{-1}_E(E(\cO_K))$ is $A[G]$--injective.
Thus, the exact sequence in question splits in the category of $A[G]$--modules. The corollary follows immediately. (See proof of \cite[Cor. 6.1.2]{FGHP20} for details.)
\end{proof}

\subsection{Refined Brumer-Stark for t-Modules}\label{BrSt-section}
The theorem below is the $t$--module analogue of the classical Brumer--Stark conjecture for class--groups of number fields. The reader is encouraged to read \S6.2 in \cite{FGHP20} for detailed comments on
the classical Brumer--Stark conjecture and its Drinfeld module analogue proved in loc.cit.
\begin{theorem}[Refined Brumer--Stark for $t$--modules]\label{T:Main-BrSt}  Let $\cM$ be taming module for $K/F$ and let $E$ be an abelian $t$-module with structural morphism $\phi_E:\Fq[t]\to\Mat_n(\cO_F)\{\tau\}$. Let $\Lambda'$ be a $E(K_\infty)/E(\cM)$--admissible $A[G]$--lattice  in $\Lie_E(K_\infty)$ (as in \ref{D:Admissible lattice}), then we have
$$\frac{1}{[\Lie_E(\cM):\Lambda']_G}\cdot \Theta_{K/F}^{E, \cM}(0)\in{\rm Fitt}^0_{A[G]}H(E/\cM).$$
\end{theorem}

\begin{proof}
Let $s:H(E/\cM) \to E(K_\infty)/E(\cM)$ be an $\F_q[G]$-linear splitting for the structural exact sequence \eqref{E:ExactSequenceM_1}, such that $\Lambda'$ is $s$--admissible. From the definition of the volume function and from Theorem \ref{T:ETNFtMods} we have
\[\frac{1}{[\Lie_E(\cM):\Lambda']_G}\cdot \Theta_{K/F}^{E, \cM}(0)= \left|\Lambda'/{\rm exp_E}^{-1}(E(\cM))\times s(H(E/\cM))\right|_{G}.\]
Then, since $|\cdot|_G$ is defined to be the generator of the Fitting ideal, we deduce
\[\frac{1}{[\Lie_E(\cM):\Lambda']_G}\cdot \Theta_{K/F}^{E, \cM}(0)\in {\rm Fitt}^0_{A[G]}\left(\Lambda'/{\rm exp_E}^{-1}(E(\cM))\times s(H(E/\cM))\right).\]
Then, observe that from \eqref{E:ExactSequenceM_1} we have an $A[G]$--linear surjection
$$(\Lambda'/{\rm exp_E}^{-1}(E(\cM))\times s(H(E/\cM))\twoheadrightarrow H(E/\cM),$$
which, by basic properties of Fitting ideals gives an inclusion
\[{\rm Fitt}^0_{A[G]}\left(\Lambda'/{\rm exp_E}^{-1}(E(\cM))\times s(H(E/\cM))\right)\subseteq {\rm Fitt}^0_{A[G]}H(E/\cM)),\]
from which the theorem follows.
\end{proof}

As a consequence of the above theorem, we obtain two corollaries.

\begin{corollary}\label{C:Main-BrSt} With notations as in Theorem \ref{T:Main-BrSt}, we have
$$\frac{1}{[\Lie_E(\cM):\Lambda']_G}\cdot \Theta_{K/F}^{E, \cM}(0)\in{\rm Fitt}^0_{A[G]}H(E/\cO_K).$$
\end{corollary}
\begin{proof} Since $\cM \subset \cO_K$, we have a natural surjective morphism $H(E/\cM)\twoheadrightarrow H(E/\cO_K)$ of $A[G]$--modules. This gives an inclusion ${\rm Fitt}^0_{A[G]}H(E/\cM)\subseteq {\rm Fitt}^0_{A[G]}H(E/\cO_K).$
\end{proof}

\begin{corollary}\label{C:Main-fullFitt}
If $p\nmid |G|$, then we have an equality of principal $A[G]$--ideals
$$\frac{1}{[\Lie_E(\cO_K): {\rm exp}^{-1}_E(E(\cO_K))]_G}\Theta_{K/F}^E(0)\cdot A[G]={\rm Fitt}^0_{A[G]} H(E/\cO_K).$$
\end{corollary}
\begin{proof}
This follows directly from Corollary \ref{main-theorem-not-p}
\end{proof}

\subsection{Imprimitive (Euler incomplete) $L$--values}\label{incomplete-L-section} In this section we explain how the results in \S\ref{ETNF-section} above lead to equivariant Tamagawa number formulas for any Euler--incomplete special $L$--value of the type
$$\Theta_{K/F, S}^E(0):=\prod _{v\not\in S}\frac{\vert E(\cO_K/v)\vert_G}{\vert \Lie_E(\cO_K/v)\vert_G},$$
where $S$ is any finite subset of ${\rm MSpec}(\cO_F)$ containing the wild ramification locus $W$ for the ring extension $\cO_K/\cO_F$.

\begin{definition}\label{D:supertaming} Let $S$ be a set as above. An $\cO_F\{\tau\}[G]$--submodule $\cM$ of $\cO_K$ is called $(E, S, \cO_K/\cO_F)$--taming if the following hold.
\begin{enumerate}
\item $\cM$ is $(S, \cO_K/\cO_F)$--taming, i.e. $\cM$ is $\cO_F[G]$--projective and $\cO_K/\cM$ is finite and supported only at primes in $S$.
\item For any $v\in S$, we have an isomorphism of $A[G]$--modules
$$E(\cM/v)\simeq \Lie_E(\cM/v).$$
\end{enumerate}
\end{definition}
\begin{remark}\label{R:supertaming} Observe that for an $(E, S, \cO_K/\cO_F)$--taming module $\cM$ and any $v\in S$, we have an equality of monic elements in $A[G]$
$$\vert E(\cM/v)\vert_G =\vert \Lie_E(\cM/v)\vert_G.$$
Consequently, from the definitions we have an equality of special values
$$\Theta_{K/F, S}^E(0)=\Theta_{K/F}^{E, \cM}(0).$$
In particular, if $S=W$, then $\cM$ is an $\cO_K/\cO_F$--taming module in the classical sense, with the additional property (2) in the definition above, so we get equalities
$$\Theta_{K/F, W}^E(0)=\Theta_{K/F}^{E, \cM}(0).$$
\end{remark}
\medskip

\noindent The next Lemma builds upon ideas developed in the non--equivariant setting in \S3.2 of \cite{ANT} and it  gives us a recipe for constructing $(E, S, \cO_K/\cO_F)$--taming modules. It is very important to note that the recipe is independent of the $t$--module $E$, i.e. it produces 
$(E, S, \cO_K/\cO_F)$--taming modules, for all $t$--modules $E$ defined over $\mathcal O_F$.
\begin{lemma}\label{L:supertaming} Let $S$ be a set as above, let $\cM$ be an $\cO_K/\cO_F$--taming module and let
$$\xi\in\prod_{v_0\in S_A}v_0,\qquad \cM_\xi:=\xi\cM,$$
where $S_A$ is the subset of ${\rm MSpec}(A)$ consisting of all the primes sitting below primes in $S$.
Then $\cM_\xi$ is an $(E, S, \cO_K/\cO_F)$--taming module, for all $t$--modules $E$ defined over $\mathcal O_F$.
\end{lemma}
\begin{proof} The fact that $\cM_\xi$ satisfies condition (2) in Definition \ref{D:supertaming} is an immediate consequence of the $\cO_F[G]$--module isomorphism $\cM\simeq\cM_\xi$ given by multiplication by $\xi$ 
and the observation that $\cM/\xi\cM$ is finite, supported only at primes in $S$.
 
Now, it is easy to see that for all $x\in (\cM_\xi)^n$, all $a\geq 1$, and all $v\in S$, we have 
$$\tau^a(x)\in \xi^{q^a-1}\cdot (\xi\cM)^n\subseteq (v\cM_\xi)^n.$$
Consequently, for all $\alpha\in A$ and all $\widehat x\in (\cM_\xi/v\cM_\xi)^n$ , we have
$$d_E[\alpha](\widehat x)=\varphi_E(\alpha)(\widehat x).$$
This proves that for all $v\in S$, the identity map gives an $A[G]$--linear isomorphism
 $$E(\cM/v)\simeq \Lie_E(\cM/v),$$
which concludes the proof of the Lemma.
\end{proof}

\begin{definition}\label{univ-taming-definition} 
An $\mathcal O_F[G]$--submodule $\mathcal M$ of $\mathcal O_K$  is called $(S, \mathcal O_K/\mathcal O_F)$--universally taming if it is $(E, S, \mathcal O_K/\mathcal O_F)$--taming, for all $t$--modules $E$ as above.
\end{definition}
\begin{theorem}[imprimitive ETNF]\label{T:imprimitiveETNF}
Let $E$ be an abelian $t$-module as above. Let $S$ be a finite subset of ${\rm MSpec}(\mathcal O_F)$, such that $W\subseteq S$. Let  $\cM$ be any $(E, S, \cO_K/\cO_F)$--taming module. Then, we have the following equality in $(1+t^{-1}\F_q[[t^{-1}]][G])$.
\[\Theta_{K/F, S}^{E}(0)=\frac{{\rm Vol}(E(K_\infty)/E(\cM))}{{\rm Vol}(\Lie_E(K_\infty)/\Lie_E(\cM))}.\]
\end{theorem}
\begin{proof}
For $S=W$, this is a direct consequence of Theorem \ref{T:ETNFtMods} and Remark \ref{R:supertaming}.
For a more general $S$,  it is not difficult to see that the techniques in this paper can be used to prove Theorem \ref{T:ETNFtMods} for $(E, S, \cO_K/\cO_F)$--taming modules $\cM$.  
As a consequence, one obtains an ETNF as in Theorem \ref{T:imprimitiveETNF} for the more general imprimitive $L$--values $\Theta_{K/F, S}^E(0)$. We leave this as an exercise to the interested reader.
\end{proof}

\subsection{$L$--values at positive integers}\label{positive-integers-section} Now, assume that $E$ is a Drinfeld $t$--module (i.e. $n=1$) of rank $r$, defined over $\mathcal O_F$ and let $\mathcal C$ be the (rank $1$) Carlitz module given by the structural morphism
$$\phi_{\mathcal C}:\Bbb F_q[t]\to \mathcal O_F\{\tau\}, \qquad \phi_{\mathcal C}(t)=t\tau^0+\tau.$$
For $m\in\Bbb Z_{\geq 0}$, we consider the $(1+m\cdot r)$--dimensional, abelian $t$--module defined over $\mathcal O_F$:
$$E(m):=E\otimes\mathcal C^{\otimes m}.$$
(See \S\S5.7--5.8 in \cite{Goss} for the definition and properties of tensor products of Drinfeld modules in the category of abelian $t$--modules and keep in mind that Drinfeld modules are pure $t$--modules in the terminology of loc.cit.) 
\begin{lemma}
With notations as above, we have an equality of values of Euler--incomplete equivariant $L$--functions
$$\Theta_{K/F, S}^{E(m)}(0)=\Theta_{K/F, S}^E(m),\qquad\text{ for all }m\in\Bbb Z_{\geq 0}$$
and all finite sets $S\subseteq {\rm MSpec}(\mathcal O_F)$ containing the set $S_0$ of primes of bad reduction for $E$ and wildly ramfied primes for $\mathcal O_K/\mathcal O_F$.
\end{lemma}
\begin{proof} We start by noting that the Carlitz module $\mathcal C$ has good reduction at all primes in ${\rm MSpec}(\mathcal O_F)$ and, consequently, the bad reduction loci of $E(m)$ and $E$ coincide for all $m$.
Now, we refer the reader to the notations introduced at the very beginning of \S\ref{L-values-section}. Let $v_0\in{\rm MSpec}(A)$.  Since the $v_0$--adic realization functor $H^1_{v_0}(\ast)$ commutes with tensor products 
(see Proposition 5.7.3 in \cite{Goss}), we have an isomorphism of $A_{v_0}[G]$--modules
$$H_{v_0}^1(E(m))\simeq H_{v_0}^1(E)\otimes_{A_{v_0}}H_{v_0}^1(\mathcal C)^{\otimes m}, $$
for all $m$ as above. Now, take a prime $v\in{\rm MSpec}(\mathcal O_F)\setminus S$, such that $v\nmid v_0$. 
It is well--known that the Frobenius $\widetilde{\sigma_v}$ acts via multiplication by $(Nv)^{-1}$ on $H_{v_0}^1(\mathcal C)$. (See Corollary 2.5 in \cite{Hayes}, for example.)  Therefore, from the definitions we have
$$P_v^{\ast, G, E(m)}(Nv^{-s})=P_v^{\ast, G, E}(Nv^{-(s+m)}), \qquad \text{  for all }s\in\Bbb S_\infty.$$
By setting $s=0$ above, we  conclude from \eqref{incomplete-L-definition} and Remark \ref{actual-values-remark} (noting once again that $E(m)$ is a pure abelian $t$--module) that we have 
\begin{equation}\label{E:theta(m)}\Theta_{K/F, S}^{E(m)}(0)=\Theta_{K/F, S}^E(m), \qquad\text{ for all }m\in\Bbb Z_{\geq 0}, \end{equation}
which concludes the proof of the Lemma.
\end{proof}

In light of the above Lemma, Theorem \ref{T:imprimitiveETNF} implies the following equivariant Tamagawa number formula for values of $G$--equivariant $L$--functions of Drinfeld modules at all positive integers $m$.
\begin{theorem}[ETNF at positive integers]\label{T:ETNF-positive-integers} Let $E$ be a Drinfeld module of structural morphism 
$\phi_E:A\to\mathcal O_F\{\tau\}$. Let $S$ be a finite set of primes in ${\rm MSpec}(\mathcal O_F)$, containing the primes of bad reduction for $E$ and the wildly ramified primes in $\mathcal O_K/\mathcal O_F$.
Let $\cM$ be a $(\cO_K/\cO_F, S)$--universally taming module. Then, we have the following equalities in $(1+t^{-1}\F_q[[t^{-1}]][G])$
\[\Theta_{K/F, S}^{E}(m)=\frac{{\rm Vol}(E(m)(K_\infty)/E(m)(\cM))}{{\rm Vol}(\Lie_{E(m)}(K_\infty)/\Lie_{E(m)}(\cM))},\]
for all $m\in\Bbb Z_{\geq 0}.$
\end{theorem}
\begin{proof}
Combine Theorem \ref{T:imprimitiveETNF}  with equalities \eqref{E:theta(m)}.
\end{proof}

In the classical (number field) theory of $G$--equivariant Artin $L$--functions associated to an abelian extension $K/F$ of number fields of Galois group $G$, the $L$--value at $s=0$ is linked via the Refined Brumer--Stark Conjecture to Fitting ideals of class group ${\rm Cl}(\mathcal O_K)$, which should be viewed as the torsion part of 
the Quillen $K$--group $K_0(\mathcal O_K)$. At the same time, the $G$--equivariant $L$--values at negative 
integers $s=(1-m)$ are linked via the Refined Coates--Sinnott  Conjecture to Fitting ideals of the even Quillen $K$--groups $K_{2m}(\mathcal O_K)$. The reader may consult \cite{GP15} for the precise statements and conditional proofs of these classical conjectures in number theory.\\

The following is the Drinfeld module analogue of the Refined Coates--Sinnott Conjecture.

\begin{theorem}[Refined Coates--Sinnott for Drinfeld modules]\label{T:Main-CoatesSinnott}  For data as in Theorem \ref{T:ETNF-positive-integers},  let $\Lambda'$ be a $E(m)(K_\infty)/E(m)(\cM)$--admissible $A[G]$--lattice  in $\Lie_{E(m)}(K_\infty)$. Then
$$\frac{1}{[\Lie_{E(m)}(\cM):\Lambda']_G}\cdot \Theta_{K/F, S}^{E}(m)\in{\rm Fitt}^0_{A[G]}H(E(m)/\cM)\subseteq {\rm Fitt}^0_{A[G]}H(E(m)/\cO_K),$$
for all $m\in\Bbb Z_{\geq 0}$.
\end{theorem}
\begin{proof}
Combine Theorem \ref{T:ETNF-positive-integers} with the proofs of Theorem \ref{T:Main-BrSt} and Corollary \ref{C:Main-BrSt}.
\end{proof}

In light of the number field facts briefly summarized above, the last Corollary suggests that if we are to think of Taelman's class module $H(E/\mathcal O_K)$ as an analogue of a class group for the data $(E,\mathcal O_K/\mathcal O_F)$, then it is natural
to think of $H(E(m)/\mathcal O_K)$ is an analogue of a higher Quillen $K$--group for the given data, for all $m\in\Bbb Z_{\geq 0}$. These analogies will be explored further in an upcoming paper, where a development of an Iwasawa Theory for Taelman's class modules will be attempted.

\begin{remark} All the results in the current section hold true and are proved in the same way for an arbitrary rank {\it pure abelian $t$--module} $E$. (See \cite{Goss}, \S5.5 for the definition of purity and \S\S5.7-5.8 for properties of tensor products of pure abelian $t$--modules, while keeping in mind that there is an antiequivalence of categories between the category of $t$--modules and that of $t$--motives, as Theorem 5.4.11 in loc.cit. shows.)
\end{remark}

\bibliographystyle{amsplain}
\bibliography{ETNFBib}

\providecommand{\bysame}{\leavevmode\hbox to3em{\hrulefill}\thinspace}
\providecommand{\MR}{\relax\ifhmode\unskip\space\fi MR }
\providecommand{\MRhref}[2]{%
  \href{http://www.ams.org/mathscinet-getitem?mr=#1}{#2}
}
\providecommand{\href}[2]{#2}
\begin{thebibliography}{10}

\bibitem{And86}
Greg~W. Anderson, \emph{{$t$}-motives}, Duke Math. J. \textbf{53} (1986),
  no.~2, 457--502. \MR{850546}

\bibitem{Anderson00}
\bysame, \emph{An elementary approach to {$L$}-functions mod {$p$}}, J. Number
  Theory \textbf{80} (2000), no.~2, 291--303. \MR{1740516}

\bibitem{ANT}
Bruno Angl\`es, Tuan Ngo~Dac, and Floric Tavares~Ribeiro, \emph{A class formula
  for admissible {A}nderson modules}, Invent. Math. \textbf{229} (2022), no.~2,
  563--606. \MR{4448991}

\bibitem{Beaumont}
Tiphaine Beaumont, \emph{On equivariant class number formulas for {A}nderson
  modules}, Res. Number Theory \textbf{9} (2023), no.~4, Paper No. 68, 33.
  \MR{4646438}

\bibitem{Dem}
Florent Demeslay, \emph{A class formula for {$L$}-series in positive
  characteristic}, Ann. Inst. Fourier (Grenoble) \textbf{72} (2022), no.~3,
  1149--1183. \MR{4485822}

\bibitem{Fang22}
Jiangxue Fang, \emph{Equivariant special {$L$}-values of abelian
  {$t$}-modules}, J. Number Theory \textbf{232} (2022), 242--2060. \MR{3276327}

\bibitem{FGHP20}
Joseph Ferrara, Nathan Green, Zach Higgins, and Cristian~D. Popescu, \emph{An
  equivariant {T}amagawa number formula for {D}rinfeld modules and
  applications}, Algebra Number Theory \textbf{16} (2022), no.~9, 2215--2264.
  \MR{4523328}

\bibitem{Goss}
David Goss, \emph{Basic structures of function field arithmetic}, Ergebnisse
  der Mathematik und ihrer Grenzgebiete (3) [Results in Mathematics and Related
  Areas (3)], vol.~35, Springer-Verlag, Berlin, 1996. \MR{1423131}

\bibitem{GP15}
Cornelius Greither and Cristian~D. Popescu, \emph{An equivariant main
  conjecture in {I}wasawa theory and applications}, J. Algebraic Geom.
  \textbf{24} (2015), no.~4, 629--692. \MR{3383600}

\bibitem{Hayes}
D.~R. Hayes, \emph{Explicit class field theory for rational function fields},
  Trans. Amer. Math. Soc. \textbf{189} (1974), 77--91. \MR{330106}

\bibitem{Igusa}
Jun-ichi Igusa, \emph{An introduction to the theory of local zeta functions},
  AMS/IP Studies in Advanced Mathematics, vol.~14, American Mathematical
  Society, Providence, RI; International Press, Cambridge, MA, 2000.
  \MR{1743467}

\bibitem{PR24}
Cristian~D. Popescu and N.~Ramachandran, \emph{Euler factors of equivariant
  {$L$}--functions of {D}rinfeld modules and beyond},
  https://arxiv.org/abs/2406.13976v1 (2024), 17 pages.

\bibitem{T12}
Lenny Taelman, \emph{Special {$L$}-values of {D}rinfeld modules}, Ann. of Math.
  (2) \textbf{175} (2012), no.~1, 369--391. \MR{2874646}

\end{thebibliography}

\end{document}